\documentclass[]{article}
\usepackage{graphicx,xcolor,cite,mathtools,bbold,bbding}
\usepackage{amssymb,amsmath,amsthm,mathrsfs,paralist,bm,esint,setspace}
\usepackage[ngerman,french,english]{babel}
\RequirePackage[colorlinks,citecolor=blue,urlcolor=blue]{hyperref}
\usepackage{cases}
\usepackage{color}
\usepackage{setspace}
\let\OLDthebibliography\thebibliography
\renewcommand\thebibliography[1]{
  \OLDthebibliography{#1}
  \setlength{\parskip}{3pt}
  \setlength{\itemsep}{0pt plus 0.3ex}
}

\DeclareMathOperator{\Var}{Var}

\newcommand{\N}{\mathbb{N}}
\newcommand{\cN}{\mathcal{N}}
\newcommand{\Z}{\mathbb{Z}}

\newcommand{\R}{\mathbb{R}}

\newcommand{\F}{\mathcal{F}}
\newcommand{\T}{\mathbb{T}}

\newcommand{\lip}{\text{\rm Lip}}

\renewcommand{\P}{\mathrm{P}}
\newcommand{\E}{\mathrm{E}}

\renewcommand{\d}{{\rm d}}

\newcommand{\e}{{\rm e}}
\renewcommand{\geq}{\geqslant}

\renewcommand{\ge}{\geqslant}
\renewcommand{\le}{\leqslant}

\author{Davar Khoshnevisan\\University of Utah
	\and Kunwoo Kim\\POSTECH
	\and Carl Mueller\\University of Rochester
	\\[.5cm]\and
	\emph{In the memory of Giuseppe Da Prato}\\
	\emph{(July 23, 1936--October 5, 2023)}
	}
\title{\bf Small-ball constants, and \\exceptional flat points of SPDEs\thanks{
	Research supported in part by the NSF grant DMS-2245242 [D.K.],
	NRF grants  2019R1A5A1028324, 2021R1A6A1A10042944, and RS-2023-00244382  [K.K.], 
	and Simons Foundation grant 513424 [C.M.].  
}}
\date{December 9, 2023}
\newtheorem{stat}{Statement}[section]
\newtheorem{proposition}[stat]{Proposition}

\newtheorem{corollary}[stat]{Corollary}
\newtheorem{theorem}[stat]{Theorem}
\newtheorem{lemma}[stat]{Lemma}
\theoremstyle{definition}

\newtheorem{remark}[stat]{Remark}
\newtheorem{OP}{Open Problem}

\numberwithin{equation}{section}

\begin{document}
\maketitle
\begin{abstract} 
	We study small-ball probabilities for the stochastic heat equation with 
	multiplicative noise in the moderate-deviations regime. We prove the 
	existence of a small-ball constant and related it to other known quantities 
	in the literature. These small-ball estimates are known to imply Chung-type
	laws of the iterated logarithm (LIL) at typical spatial points; these points
	can be thought of as ``points of flat growth.'' For this result in a similar context in 
	SPDEs see, for example, the recent work of
	Chen \cite{Ch2023}. We establish the existence of a new 
	family of exceptional spatial points where the Chung-type LIL fails.  
\end{abstract}

\noindent{\it Keywords:} Stochastic partial differential equations, small-ball
	probabilities, flat points.\\
\noindent{\it \noindent AMS 2020 subject classification:}
	Primary: 60H15,  Secondary: 60G17, 60F99.
\selectlanguage{english}


\section{Introduction and main results}

Let $X=\{X(t)\}_{t\in\mathcal{T}}$ be a real-valued stochastic process with continuous sample functions,
where $\mathcal{T}$
is a compact, separable metric space. By a small-ball probability estimate we mean
an approximation of $\log\P\{\sup_{t\in \mathcal{T}}|X(t)|\le \varepsilon\}$
that is ideally valid  uniformly for all small
$\varepsilon$ (say $0<\varepsilon<1$). We seek to find
asymptotic bounds, and
the set $\mathcal{T}$ can also depend on the parameter $\varepsilon$.
Such results were first developed by Chung \cite{C1948}
for the simple walk and for Brownian motion on $\R$, in order to prove
so-called Chung-type laws of the iterated logarithm (LIL). More specifically, Chung's work \cite{C1948}
for a 1-dimensional Brownian motion $X$ (set $\mathcal{T}=\mathcal{T}(\varepsilon)=[0\,,\varepsilon]$ with
the usual Euclidean distance) implies that, with probability one,
\[
	\liminf_{\varepsilon\downarrow0}\left(\frac{\log|\log\varepsilon|}{\varepsilon}\right)^{1/2}
	\sup_{s\in[0,\varepsilon]}|X(s)|=\frac{\pi}{\sqrt 8}.
\]
The literature on small-ball probabilities and Chung-type LILs has since grown considerably;
see the survey paper of Li and Shao \cite{LS2001} for the development of the theory
up to earlier 2000s in the context of Gaussian processes.
Dereich, Fehringer, Maroussi, and Scheutzow \cite{DFMS2003},
Klartag and Vershynin \cite{KV2007}, and Kuelbs and Li 
\cite{KL1993} discuss various connections between small-ball probability estimates
and other parts of mathematics, specifically approximation theory and quantization problems 
in Banach space theory. Much of the preceding is concerned mainly with the so-called
$L^\infty$ theory for Gaussian measures. A recent survey by Nazarov and Petrova \cite{NP2023}
describes up-to-date information, particularly for the closely-related $L^2$-type theory
of small-ball estimates for Gaussian measures. Here, we pursue
aspects of some $L^\infty$-type problems for stochastic PDEs of a parabolic type.

Let $\T=[-1\,,1]\cong\R/(2\Z)$ denote the one-dimensional torus and consider the
following parabolic stochastic PDE (or SPDE) on $\R_+\times\T$:
\begin{equation}\label{u}\left[\begin{split}
	&\partial_t u (t\,,x) = \partial^2_x u(t\,,x) + \sigma(u(t\,,x))\dot{W}(t\,,x)
		& \text{for all } t>0,\ x\in\T,\\
	&\text{subject to } u(0\,,x)=u_0(x)&\text{for all $x\in\T$}.
\end{split}\right.\end{equation}
where the forcing is comprised of  a space-time white noise
$\dot{W}=\{\dot{W}(t\,,x)\}_{t\ge0\,,x\in\T}$ on $\R_+\times\T$ 
with an interaction term $\sigma:\R\to\R$ that is a non-random and Lipschitz continuous function,
and an initial data $u_0:\T\to\R$ that is non-random and Lipschitz  continuous.

Our goal is to continue the recent analyses of Athreya, Joseph, and Mueller \cite{AJM2021},
Chen \cite{Ch2023},
and Foondun, Joseph, and Kim \cite{FJK2023} to study small-ball probabilities for a nonlinear system,
such as \eqref{u}, and discuss how they relate to sample function properties of the solution to the
SPDE \eqref{u}. In the case that $\sigma$ is constant --- and in fact for much more general Gaussian random
fields that are strongly locally non deterministic --- some of this type of analysis 
was carried out by Lee and Xiao \cite{LX2023}
slightly earlier.\footnote{Small--ball probability
estimates are also available for a particular family of Gaussian processes that solve
semilinear hyperbolic SPDEs. They have a different form from the results here
and in Athreya, Joseph, and Mueller \cite{AJM2021},
Chen \cite{Ch2023},
and Foondun, Joseph, and Kim \cite{FJK2023}, and require very different methods of
analysis; see Martin \cite{M2004}, which is based in part on a celebrated earlier theorem of Talagrand
\cite{T1994b} on the small-ball problem for the Brownian sheet.}
 
These references show that,
under appropriate conditions, one can establish small-ball probability estimates
that are sharp, at the logarithmic level, up to a multiplicative constant
\cite{AJM2021,FJK2023,LX2023}.
Moreover, one can deduce Chung-type LILs for the solution to \eqref{u} under natural conditions
\cite{Ch2023,LX2023}. In this paper, we study \eqref{u} dynamically as a process
$t\mapsto u(t\,,x)$, one value of $x$ at a time, and show that:
\begin{compactenum}
\item[(1)] The resulting processes have
	a tight small-ball estimate with a more-or-less explicit small-ball constant;
	see Theorem \ref{th:u} below. This appears
	to be a first example of a family of infinite-dimensional Markov processes that have
	tight, explicit small-ball probability rates, together with identifiable small-ball
	constants; and 
\item[(2)] In addition to a more traditional Chung-type LIL (Corollary \ref{cor:Chung}), we prove that
	one can find exceptional points $x\in\T$
	at which other Chung-type LILs hold; see Theorem \ref{th:subseq}. This finding illustrates
	a new phenomenon that seems to be intimately linked to the infinite-dimensional setting, 
	and also requires novel proof ideas.
\end{compactenum}
In order to describe our results,
let $F=\{F(t)\}_{t\ge0}$ denote a \emph{fractional Brownian motion of index}
$1/4$; see Mandelbrot and VanNess \cite{MV1968}. 
That is, $F$ is a continuous, centered Gaussian process such that $F(0)=0$ and 
\[
	\E\left(|F(t) - F(s)|^2\right)=|t-s|^{1/2}\qquad\text{for all $s,t\ge0$}.
\]
We can deduce from the works of
Li \cite{L1999}, Li and Linde \cite{LL1999}, and Shao \cite{Sh2003} that 
\begin{equation}\label{LiLinde}
	\lambda=-\lim_{\varepsilon\downarrow0}\varepsilon^4
	\log\P\left\{ \sup_{t\in[0,1]}|F(t)|\le \varepsilon\right\}
	\quad\text{exists and is in $(0\,,\infty)$}.
\end{equation}
The number $\lambda$ is the so-called \emph{small-ball constant} for $F$. 
It should be possible to combine  \eqref{LiLinde} and Monte Carlo methods in order
to find a reasonable approximation to $\lambda$, but
the exact numerical value of $\lambda$ is not known.

\begin{theorem}\label{th:u}
	In addition to the preceding assumptions, suppose that $\sigma$ is bounded.
	Choose and fix an unbounded, non-increasing, deterministic  function $\phi:(0\,,1)\to(0\,,\infty)$ that 
	satisfies the following: 
	\begin{equation}\label{phi:u}
		\phi(\varepsilon) = O\left( |\!\log\varepsilon|\right)
		\quad\text{as $\varepsilon\downarrow0$}.
	\end{equation}
 Then, for every $x\in\T$,
	\[
		\lim_{\varepsilon\downarrow0}
		\frac{1}{\phi(\varepsilon)}
		\log\P\left\{ \sup_{t\in[0,\varepsilon]}|u(t\,,x)-u_0(x)| \le
		\left( \frac{\varepsilon}{\phi(\varepsilon)}\right)^{1/4}\right\}
		=-\frac{2\lambda}{\pi} \vert\sigma(u_0(x))\vert^4.
	\]
\end{theorem}

As far as we know, the first paper on small-ball probabilities was Chung \cite{C1948},
where the object of main interest was the simple walk on $\Z$ and, through that, the
Brownian motion on the line. Chung \cite{C1948} was also the first to notice
that small-ball probabilities can be used to yield a matching law of the iterated logarithm
(LIL). Thus, it should not come as a surprise that Theorem \ref{th:u} too implies 
a Chung-type LIL. Though we pause to point out that additional effort is required to show the
next corollary, as it is valid under
fewer technical hypotheses than is Theorem \ref{th:u}.

\begin{corollary}\label{cor:Chung}
	Regardless of whether or not $\sigma$ is bounded,
	\begin{equation}\label{LIL:Chung}
		\liminf_{\varepsilon\downarrow0}\left( \frac{%
		\log|\log\varepsilon|}{\varepsilon}\right)^{1/4}
		\sup_{t\in[0,\varepsilon]}|u(t\,,x)-u_0(x)| = 
		\left( \frac{2\lambda}{\pi} \right)^{1/4}|\sigma(u_0(x))|,
	\end{equation}
	a.s.\ for every $x\in\T$ , where $\lambda$ was defined in \eqref{LiLinde}.
\end{corollary}

To be sure of the order of the quantifiers, 
we note that Corollary \ref{cor:Chung} says that for every non-random point
$x\in\R$ there exists a $\P$-null set off of which \eqref{LIL:Chung} holds.
We may view such points $x$ as points of [relatively] ``flat growth,'' for example
as compared with points where iterated logarithm fluctuations are observed;
see \cite{HK2017}.
Corollary \ref{cor:Chung} and Fubini's theorem together show that the collection of all points $x\in\T$
that satisfy \eqref{LIL:Chung} has full Lebesgue/Haar measure.
The remainder of our effort is concerned with studying many of the points $x\in\T$
that are exceptional in the sense that they fail to satisfy \eqref{LIL:Chung}. 
A standard method
for finding such points is to appeal to the theory of limsup random fractals \cite{KPX2000} and adapt
it to the present small-ball setting for SPDEs. For large-ball problems, this adaptation was
done in \cite{HK2017}, 
and we feel that similar methods will yield exceptional points $x\in\T$
for which the rate $\text{const}\times
(\varepsilon^{-1}\log|\log\varepsilon|)^{1/4}$ is replaced by 
rate $\text{const}\times (\varepsilon^{-1}|\log\varepsilon|)^{1/4}$
for suitable choices of ``const.'' We have not tried to do that here. 
Instead, we document the existence of a more subtle family of exceptional points $x\in\T$
whose existence requires new proof ideas.
In order to present that family we need some notation.

From now on, we will use the symbol $\rightsquigarrow$ to denote subsequential limits.
More precisely, whenever $a,a_1,a_2,\cdots\in\R$, then we might
write ``$a_n\rightsquigarrow a$ as $n\to\infty$''
as shorthand for ``$\liminf_{n\to\infty}|a_n-a|=0$.'' 

\begin{theorem}\label{th:subseq}
	Choose and fix a non-random, nonnegative, extended real number $\chi\in[0\,,\infty]$. Then, 
	regardless of whether or not $\sigma$ is bounded, there a.s.\ exists a random $x\in\T$ such that
	\begin{equation}\label{LIL:1}
		\left( \frac{\log|\log\varepsilon|}{\varepsilon}\right)^{1/4}
		\sup_{t\in[0,\varepsilon]} |u(t\,,x)-u_0(x)|
		\rightsquigarrow \chi^{1/4} |\sigma(u_0(x))| \text{ as $n\to\infty$}  
	\end{equation}
	If $\chi\in[0\,,2\lambda/\pi]$, then there in fact a.s.\ exists a
	random $x\in\T$ such that
	\begin{equation}\label{LIL:2}
		\liminf_{\varepsilon\downarrow0}
		\left( \frac{\log|\log\varepsilon|}{\varepsilon}\right)^{1/4}
		\sup_{t\in[0,\varepsilon]} |u(t\,,x)-u_0(x)|
		= \chi^{1/4}|\sigma(u_0(x))|.
	\end{equation}
\end{theorem}

We pause to insert a few problems that have eluded us.

\begin{OP}
	Can \eqref{LIL:1} be upgraded to \eqref{LIL:2}
	when $\chi>2\lambda/\pi$?
	We suspect the answer is ``no.''
\end{OP}

\begin{OP}
	Based on an informal comparison with limsup random fractals,
	we conjecture that the set of $x\in\T$ that satisfy either condition
	\eqref{LIL:1} or \eqref{LIL:2} always has full Hausdorff dimension a.s.
\end{OP}

Above and throughout, we view $\T$ as the set $[-1\,,1]$ and identify it with the 
abelian group $\R/(2\Z)$ in the customary manner:
We use the additive notation for $\T$, and in
fact move back and forth from interpreting $\T$ as the real interval $[-1\,,1]$ to the abelian group
$\R/(2\Z)$. In particular, we write ``$x-y$'' instead of ``$x-y\pmod 2$''
or ``$xy^{-1}$'' for $x,y\in\T$, and designate $0$ (not 1) as the group
identity. We also denote
by $\d x$ an infinitesimal element of a Haar measure on $\T$ and do not distinguish between the
Lebesgue measure on $[-1\,,1]$, normalized to have total mass $2$ and the 
Haar measure on $\T$, similarly normalized.

We frequently set $\log_+(a) = \log(a\vee\exp(\e))$ for all $a\ge0$.

Suppose $A$ is a nice metric space and $g:A\to\R$ is continuous.
Then we often write $\|g\|_{C(B)}$ in place of $\sup_{x\in B}|g(x)|$
whenever $B\subseteq A$. When $B=[a\,,b]$ is a subinterval of $\R$
we might write $\|g\|_{C[a,b]}$ in place of $\|g\|_{C([a,b])}$. 
Throughout, the $L^k(\Omega)$-norm
of a random variable $Z\in L^k(\Omega)$ is denoted by
$\|Z\|_k := \{ \E\left(|Z|^k\right)\}^{1/k}$
for all $1\le k<\infty$.

Let us conclude the Introduction with an outline of this paper. In Section \ref{sec:linear}, 
we investigate  small ball probabilities for the constant-coefficient case $\sigma\equiv 1$
in \eqref{u}. In Section \ref{sec:linearization}, we consider  the linearization of the 
nonlinear equation \eqref{u} and present  detailed  estimates for the difference 
between the nonlinear one and its linearization (see Proposition \ref{pr:localize}). 
We will use these estimates, along with the results from Section \ref{sec:linear}, 
to prove Theorem \ref{th:u}.  Sections 4 and 5 are dedicated to  the proofs of   
Corollary \ref{cor:Chung} and Theorem \ref{th:subseq} by using
Theorem \ref{th:u} and introducing some novel ideas. 

\section{The linear case}\label{sec:linear}

As is commonly done \cite{DZ1992,W1986}, we interpret the SPDE \eqref{u}
as the following random integral equation:
\begin{equation}\label{u:mild}
	u(t\,,x) = (p_t*u_0)(x) +  \int_{(0,t)\times\T} p_{t-s}(x\,,y)
	\sigma(u(s\,,y))\,W(\d s\,\d y),
\end{equation}
for all $t>0$ and $x\in\T$, where $p$ denotes the heat kernel on $\T$; that is,
for all $r>0$ and $x,y\in\T$,
\begin{equation}\label{p:G}
	p_r(x\,,y) =  \sum_{n=-\infty}^\infty G_r(x-y+2n),
	\text{ where }\\
	G_r(a) = \frac{\exp\{-a^2/(4r)\}}{\sqrt{4\pi r}},
\end{equation}
for every $a\in\R$.
It is well known that in short times the solution to \eqref{u} is very close 
to a constant multiple of the solution to the following linearized version of 
\eqref{u}; see \cite{HP2015,KSXZ2013}. Therefore,
we reserve the letter $Z$ specifically for the solution to the following SPDE.
\begin{align*}
	&\partial_t Z (t\,,x) = \partial^2_x Z(t\,,x) + \dot{W}(t\,,x)
		& \text{for all } t>0,\ x\in\T,\\
	&\text{subject to } Z(0\,,x)=0&\text{for all $x\in\T$}.
\end{align*}
According to \eqref{u:mild}, we may write the solution $Z$ as the following 
Wiener integral process,
\begin{equation}\label{Z:mild}
	Z(t\,,x) = \int_{(0,t)\times\T} p_{t-s}(x\,,y)\,W(\d s\,\d y)
	\qquad\text{for all }t>0,
\end{equation}
where the kernel $p$ was defined in \eqref{p:G}. In this section, we study the 
specialization of Theorem \ref{th:u} to the Gaussian 
random field $Z$, viewed as an approximation for the process $u$.
\begin{proposition}\label{pr:Z}
	Choose and fix an unbounded, non-increasing, deterministic  function $\phi:(0\,,1)\to(0\,,\infty)$ that 
	satisfies \eqref{phi:u}.
	Then,
	\[
		\lim_{\varepsilon\to0^+}
		[\phi(\varepsilon)]^{-1}\log \P\left\{ \|Z(t)\|_{C[0,\varepsilon]} \le 
		\left( \varepsilon / \phi(\varepsilon) \right)^{1/4}
		\right\}=- 2\lambda/\pi,
	\]
	where $\lambda$ was defined in \eqref{LiLinde}.
\end{proposition}
As $Z$ is a nice Gaussian random field, we will prove Proposition \ref{pr:Z}
by following a similar route to that taken in \cite{LX2023}, and then appeal to the results 
in \cite{LL1999,LS1999,Sh2003}
in order to prove the existence of the small-ball constant
and then to identify that constant.
It should be pointed out that scaling plays a role in the  methods
of the latter three references. 
Thus, a certain amount of additional effort is expended 
in order to overcome the lack of scaling for $Z$.

\subsection{The linear heat equation on free space} 
So far, $\dot{W}(t\,,x)=\partial^2_{t,x}W(t\,,x)$ where $W$ denotes a space-time Brownian sheet
that is indexed by $(t\,,x)\in\R_+\times[-1\,,1]$. Without loss of generality, and in a standard manner,
we can extend the domain of definition of the Brownian sheet $W$ so that it is in fact a space-time Brownian 
sheet on the full space $\R_+\times\R$. This canonically extends the domain of the definition of the white noise
$\dot{W}$ to all of $\R_+\times\R$ as well. With this in mind, let us consider the stochastic heat equation,
\begin{equation}\label{H}\left[\begin{split}
	&\partial_t H (t\,,x) = \partial^2_x H(t\,,x) + \dot{W}(t\,,x)
		& \text{for all } t>0,\ x\in\R,\\
	&\text{subject to } H(0\,,x)=0&\text{for all $x\in\R$}.
\end{split}\right.\end{equation}
The solution to this SPDE is, by virtue of definition and similarly to \eqref{u:mild}, the 
following Gaussian random field which is defined as a Wiener integral process,
\begin{equation}\label{H:mild}
	H(t\,,x) = \int_{(0,t)\times\R} G_{t-s}(x\,,y)\,W(\d s\,\d y)
	\qquad\text{for all }t>0\text{ and }x\in\R,
\end{equation}
where $G$ was defined in \eqref{p:G}. The following result is a precise small-ball 
estimate for the process $H$
at a given spatial point, say $x=0$,
in terms of the same small-ball constant $\lambda$ that was introduced in
\eqref{LiLinde}.

\begin{proposition}\label{pr:H}
	$\lim_{\varepsilon\downarrow0}\varepsilon^4\log\P\{ 
	\sup_{t\in[0,1]}|H(t\,,0)|\le\varepsilon\}
	=-2\lambda/\pi.$
\end{proposition}

It is well-known that one can decompose $t\mapsto H(t\,,0)$ as
a constant multiple of a fractional Brownian motion $F$ with index $1/4$ plus a
continuous Gaussian random field $T$ that is independent of $Z$ and
has $C^\infty$ sample functions
away from $t=0$; see \cite{LN2009} (Lemma \ref{lem:T} below).  
One can expect the small-ball probability of the rougher process
$F$ to dominate that of the smoother Gaussian process $T$.  
Therefore, it remains to make this assertion
rigorous.  This effort is complicated by the fact that, near $t=0$, 
the random field $T$ is not smooth;
in fact, $T$ and $F$ are equally smooth locally near $t=0$. The crux of the argument hinges on 
estimating how quickly $T$ begins to ``look like a $C^\infty$ process,'' together with a 
suitable quantitative way to interpret the quoted sentence. This effort will be summarized in
Proposition \ref{pr:T} below.
Proposition \ref{pr:H} is proved subsequently in \S\ref{subsec:Pf:Theorem:H}.

We begin by studying an auxiliary process $T$.

\subsection{An auxiliary Gaussian process}
Let $V$ denote a one-parameter white noise on $\R$; that is, $V$ is the weak derivative of a
two-sided Brownian motion indexed by $\R$. Consider the centered
Gaussian process $T=\{T(t)\}_{t\ge0}$ that is defined by $T(0)=0$ and
\begin{equation}\label{T}
	T(t) = \frac{1}{\sqrt{2\pi}}\int_{-\infty}^\infty\left( \frac{1-\e^{-tz^2/2}}{z}\right)
	V(\d z)\qquad
	\text{for $t>0$}.
\end{equation}
We shall assume throughout that $V$ and the noise $W$ in \eqref{H} are independent.
Let us recall the following structural decomposition of $H(t)$ in terms of the
process $T$ and a fractional Brownian
motion of index $1/4$.

\begin{lemma}[Lei and Nualart \cite{LN2009}]\label{lem:T}
	The centered Gaussian process $T$ is continuous. Moreover, its restriction to 
	$[\eta\,,\infty)$ is almost surely $C^\infty$ for every  $\eta>0$. Finally,
	\[
		F(t) = \frac{H(t\,,0) + T(t)}{(2/\pi)^{1/4}}\qquad
		[t\ge0]
	\] 
	defines a standard fractional Brownian motion of
	index $1/4$.
\end{lemma}

The final part of Lemma \ref{lem:T} is proved via a direct computation of the covariance function 
of $F$. With this aim in mind,
let $d$ denote the usual canonical distance that is associated to the Gaussian process $T$; that is,
\begin{equation}\label{d}
	d(s\,,t) = \| T(t)-T(s)\|_{2}\qquad\text{for all $s,t\ge0$}.
\end{equation} 
The regularity  assertions of Lemma \ref{lem:T} were proved by showing that:
\begin{equation}\label{A}\begin{split}
	&\textbf{(a)}\ d(s\,,t) \lesssim |t-s|^{1/4}\text{ uniformly for all $s,t\ge0$}; \text{ and}\\ 
	&\textbf{(b)}\ d(s\,,t) \le C_\eta |t-s|\text{ whenever $s,t\ge\eta$},
\end{split}\end{equation}
where $C_\eta$ is a number that depends on $\eta$ but not $(s\,,t)$.
The main result of this section is the following lower bound on the small-ball probability of $T$.
Note that, in addition to the assertions in Lemma \ref{lem:T}, parts (a) and (b) of \eqref{A} show that 
while $F$ is smooth away from the origin, it scales roughly as fractional Brownian motion of index 1/4
near the origin.
Nevertheless, the following shows that 
the small-ball probability of $T$ is significantly larger than that of a fractional
Brownian motion with index $1/4$.

\begin{proposition}\label{pr:T}
	There exists a constant $L>1$ such that
	\[
		\P\left\{ \|T\|_{C[0,1]} \le r\right\}
		\ge L^{-1}\exp( -L/r)\qquad\text{for all $r>0$}.
	\]
\end{proposition}

The proof of Proposition \ref{pr:T} 
requires a few preliminary steps. The first is a careful estimate on the canonical distance that
improves \eqref{A}; it in fact the following
interpolates between \textbf{(a)} and \textbf{(b)} of
\eqref{A}.

\begin{lemma}\label{lem:d}
	There exists a number $c>0$ such that
	\[
		d(s\,,t) \le c|t-s|^{1/4}\left[  1 \wedge
		\left(\frac{|t-s|}{s\wedge t}\right)^{3/4}\right]
		\qquad\text{for all $s,t>0$}.
	\]
\end{lemma}

\begin{proof}
	The definition of the Wiener integral in \eqref{T} yields the following:
	For every $t,\varepsilon\ge0$,
	\[
		d(t+\varepsilon\,,t) =\frac{\sqrt t}{\pi}\int_0^\infty
		\left( \frac{1-\e^{-\varepsilon y^2/(2t)}}{y}\right)^2\e^{-y^2}\,\d y.
	\]
	Since $1-\exp(-c)\le 1\wedge c$ for all $c\ge0$, this yields
	\[
		[d(t+\varepsilon\,,t)]^2 \le 
		\frac{\varepsilon^2}{4\pi t^{3/2}}\int_0^{\sqrt{2t/\varepsilon}}
		y^2\e^{-y^2}\,\d y + \frac{\sqrt t}{\pi}\int_{\sqrt{2t/\varepsilon}}^\infty
		\frac{\e^{-y^2}}{y^2}\,\d y = J_1 + J_2,
	\]
	notation being clear from context.
	On one hand, uniformly for
	all $t\ge\varepsilon>0$,
	\begin{align*}
		J_1 &\le \frac{\varepsilon^2}{4\pi t^{3/2}}\int_0^\infty y^2\e^{-y^2}\,\d y
			\propto \frac{\varepsilon^2}{t^{3/2}} \qquad\text{and}\\
		J_2 &\le \frac{\sqrt t}{\pi} \int_{\sqrt{2t/\varepsilon}}^\infty
			\frac{\e^{-y^2}}{y^2}\,\d y
			\asymp \frac{\varepsilon^{3/2}}{t}\,\e^{-2t/\varepsilon}
			\lesssim \frac{\varepsilon^2}{t^{3/2}},
	\end{align*}
	where we have appealed to l'H\^opital's rule to estimate $J_2$,
	as well as the fact that $A^{1/2}\exp(-A)\lesssim1$ uniformly for all $A\ge1$. 
	On the other hand, when $\varepsilon\ge t$, we have 
	\[
		J_1 \le \frac{\varepsilon^2}{4\pi t^{3/2}}\int_0^{\sqrt{2t/\varepsilon}}
		\!\!y^2\,\d y\lesssim\sqrt\varepsilon,\ 
		J_2 \le \frac{\sqrt t}{\pi}\int_{\sqrt{2t/\varepsilon}}^\infty
		\frac{\e^{-y^2}}{y^2}\,\d y 
		\lesssim \sqrt{t} \int_{\sqrt{2t/\varepsilon}}^\infty  \frac{\d y}{y^2}
		\lesssim\sqrt\varepsilon,
	\]
	valid uniformly for all $\varepsilon\ge t >0$. 
	The lemma follows from putting
	together  the two cases.
\end{proof}

We plan to use Lemma \ref{lem:d} to compute a sharp metric entropy bound
for the process $T$ on $[0\,,\varepsilon]$. In order to do that, we will need
a good covering method which will turn out to depend on the solution
to a nice difference equation.
Choose and fix a number $c>0$ and consider the initial-value problem,
\[
	g' =cg^{3/4}\quad\text{on $(0\,,\infty)$, \quad subject to $g(0)=0$},
\]
whose only increasing solution is $g(t) = (ct/4)^4$.
The following is an asymptotically analogous result for a discrete version of the preceding
ODE.

\begin{proposition}\label{pr:a}
	Choose and fix some $c>0$ and define $a_1=1$ and
	$a_{n+1}=a_n+ca_n^{3/4}$ for every $n\in\N$. Then,
	$a_n\sim(cn/4)^4$ as $n\to\infty$ in $\N$.
\end{proposition}

\begin{proof}
	By induction, $a_{n+1}\ge a_n$ for all $n\in\N$. 
	We first show that $\lim_{m\to\infty}a_m=\infty$. Indeed,
	$a_n\ge a_1=1$ for all $n\in\N$ and hence
	$a_{n+1}-a_n=ca_n^{3/4}\ge c.$ This proves the sub-optimal
	result that $a_n\ge c n$ for all $n\in\N$,
	which is nevertheless good enough to ensure that $a_n\to\infty$ as $n\to\infty$.
	
	Now we extend the sequence $\{a_n\}_{n\in\N}$ to a function $f:[1\,,\infty)\to[1\,,\infty)$
	by linear interpolation. Specifically, let
	\[
		f(t) =  a_{\lfloor t\rfloor} + c(t-\lfloor t\rfloor)a_{\lfloor t\rfloor}^{3/4} 
		\qquad\text{for all $t\ge1$},
	\]
	where $\lfloor t\rfloor$ denotes the greatest integer $\le t$. Note that
	$f$ is differentiable on $(1\,,\infty)\setminus\N$, and
	\begin{equation}\label{ODE}
		f'(s) = ca_{\lfloor s\rfloor}^{3/4} = 
		c [h(s)]^{3/4}[f(s)]^{3/4}\qquad\text{%
		for all $s\in(1\,,\infty) \setminus\N$},
	\end{equation}
	where
	\[
		h(t) = \frac{a_{\lfloor t\rfloor}}{f(t)} = \frac{a_{\lfloor t\rfloor}}{a_{\lfloor t\rfloor} 
		+ c(t-\lfloor t\rfloor)a_{\lfloor t\rfloor}^{3/4}}
		\qquad\text{for every $t\ge1$}.
	\]
	Since $0\le t-\lfloor t\rfloor\le 1$ and $a_{\lfloor t\rfloor}\to\infty$ 
	as $t\to\infty$, we have $h(t)\to 1$ 
	boundedly as $t\to\infty$.
	And of course $h(t)\le 1$ for all $t\ge1$. 
	We can write \eqref{ODE} as $\d f/f^{3/4} = ch^{3/4}\,\d s$ 
	and integrate from $1$ to $t$ [$\d s$] in
	order to find that
	\[
		f(t) = \left( 1 + \frac c4\int_1^t [h(s)]^{3/4}\,\d s\right)^4,
	\]
	for every $t\in[1\,,\infty)\setminus\N$ and hence every $t\ge1$ by
	continuity.
	Since $f(n)=a_n$ for all $n\in\N$ and $h(s)\to1$ boundedly as $s\to\infty$, this proves the result.
\end{proof}

Recall the Gaussian process $T$ and associated intrinsic metric $d$ respectively from
\eqref{T} and \eqref{d}. Let $\cN$ denote the
metric entropy of the process $\{T(t)\}_{t\in[0,1]}$. That is, for every $r>0$,
define $\cN(r)$ to be the smallest collection of open $d$-balls of radius $r>0$ needed to cover
the closed interval $[0\,,1]$. We shall recall the following result which is stated explicitly in Talagrand 
\cite[Lemma 2.2]{T1995}, whose proof follows from combining the entropy estimates of 
Talagrand \cite[Section 3]{T1994a} together with  a deep theorem of Kuelbs and Li \cite{KL1993}. A
detailed concrete proof of the following 
can be found in Section 7 of the lecture notes by Ledoux \cite{L1996}.\footnote{In fact, 
the 2-parameter process $(T\ominus T)(t\,,s) = T(t)-T(s)$ satisfies
$\P\{\|T\ominus T\|_{C([0,1]^2)}\le\varepsilon\}\ge K^{-1}\exp\{-K\psi(\varepsilon)\}.$
Lemma \ref{lem:small-ball:T} follows from this formulation since $T(0)=0$.
}

\begin{lemma}[Talagrand \cite{T1994a}]\label{lem:small-ball:T}
	Suppose $\cN \le \psi$ on $(0\,,1)$ for a function $\psi:(0\,,1)\to\R_+$ that satisfies
	$\psi(r)\asymp\psi(r/2)$  uniformly for all $r\in(0\,,1)$.
	Then there exists $K>0$ such that
	$\P \{ \|T\|_{C[0,1]} \le \varepsilon \} \ge K^{-1}\exp(-K\psi(\varepsilon))$
	for all $\varepsilon>0$.
\end{lemma}

Armed with Lemma \ref{lem:small-ball:T}, we can present the following.

\begin{proof}[Proof of Proposition \ref{pr:T}]
	It suffices to prove that the asserted inequality of the proposition is valid for all 
	$\varepsilon\in(0\,,1/2)$.
	With that aim in mind, let us
	choose and fix a real number $\varepsilon\in(0\,,1)$ [N.B.: not $\varepsilon\in(0\,,1/2)$], and define
	$a_0=0$, $a_1=1$. Then define iteratively 
	$a_{j+1}=a_j + ca_j^{3/4}$ for all $j\in\N$, where $c>0$ was
	defined in Lemma \ref{lem:d}. Also define
	\[
		t_j = a_j (  2\varepsilon/c)^4\qquad\text{for $j\in\Z_+$}.
	\]
	According to Lemma \ref{lem:d},
	$d(t_0\,,t_1)\le 2\varepsilon$, and
	\[
		d(t_j\,,t_{j+1}) \le  c|t_{j+1}-t_j| /t_j^{3/4} =2\varepsilon
		\qquad\text{for all $j\in\N$}.
	\]
	In other words, $d(t_j\,,t_{j+1})\le 2\varepsilon$ for all $j\in\Z_+$.
	It follows readily from this that
	$\cN(\varepsilon) \le 1+\max \{j\ge 0:\, a_j\le  ( c/(2\varepsilon))^4 \},$
	uniformly for all $r\in(0\,,1)$. Proposition \ref{pr:a}
	assures us that $a_j\gtrsim j^4$ uniformly for all $j\in\N$
	large. Therefore, we can see that there exists $C>0$ such that
	$\cN(\varepsilon) \le C/\varepsilon$ uniformly for every $\varepsilon\in(0\,,1)$.
	Apply Lemma \ref{lem:small-ball:T} with 
	$\psi(\varepsilon) = C/\varepsilon$ to conclude the proof.
\end{proof}

\subsection{Proof of Proposition \ref{pr:H}}\label{subsec:Pf:Theorem:H}

With the results of the preceding subsections under way, we are ready to verify 
Proposition \ref{pr:H}. But first we pause to recall the following specialization
of  \cite{A1955}.

\begin{lemma}[Anderson \cite{A1955}]\label{lem:Anderson}
	If $X$  is a  centered Gaussian random variable $X$ with values in $C[0\,,1]$, then
	$\P\{ \|X+f\|_{C[0,1]}\le r\} \le \P\{ \|X\|_{C[0,1]}\le r\}$
	for every $f\in C[0\,,1]$ and $r>0$.
\end{lemma}

Define $T$ as was done in \eqref{T}, using a noise $V$
that is independent of $W$ and hence also the solution $H$ to 
\eqref{H}. Let $F$ be the corresponding fractional Brownian motion
with index $1/4$, as was introduced in Lemma \ref{lem:T}.
Because $H$ and $T$ are independent processes, we first condition on $T$
and then appeal to 
Anderson's inequality (Lemma \ref{lem:Anderson}) in order to see that,
for all $\varepsilon>0$,
\begin{equation}\label{H:LB}\begin{split}
	\P\left\{ \|F\|_{C[0,1]}\le  (\pi/2 )^{1/4}\varepsilon \right\}
		&\le \sup_{f\in C[0,1]}\P\left\{ 
		\| H(\cdot\,,0)+f\|_{C[0,1]} \le \varepsilon\right\}\\
	&= \P\left\{ \| H(\cdot\,,0)\|_{C[0,1]} \le \varepsilon\right\}.
\end{split}\end{equation}
This yields a lower bound on the small-ball probability for $t\mapsto H(t\,,0)$
in terms of the better-studied small-ball probability for fractional Brownian motion.
For a complementary inequality let us choose and fix some number $\rho\in(0\,,1)$ and observe from
the independence of $H$ and $T$ that for all $\varepsilon>0$,
\[
	\P\left\{ \| H(\cdot\,,0)\|_{C[0,1]}\le\rho \varepsilon\right\}\cdot
	\P\left\{ \|T\|_{C[0,1]} \le (1-\rho)\varepsilon\right\}
	\le \P\left\{ \|F\|_{C[0,1]} \le (\pi/2 )^{1/4}\varepsilon \right\}.
\]
Apply Proposition \ref{pr:T} with $r=(1-\rho)\varepsilon$ in order to find a number $L>0$
such that for all $\varepsilon>0$,
\begin{equation}\label{H:UB}
	\P\left\{ \| H(\cdot\,,0)\|_{C[0,1]} \le\rho \varepsilon\right\}
	\le L\P\left\{ \|F\|_{C[0,1]} \le (\pi/2 )^{1/4}\varepsilon \right\}\cdot
	\e^{L/[(1-\rho)\varepsilon]}.
\end{equation}
Relabel $\varepsilon$ as $\rho\varepsilon$ in order to see from \eqref{H:LB} and \eqref{H:UB} that
for all $\varepsilon>0$,
\begin{gather*}
	\P\left\{ \|F\|_{C[0,1]} \le  (\pi/2 )^{1/4}\varepsilon\right\}\le
		\P\left\{ \| H(\cdot\,,0)\|_{C[0,1]} \le \varepsilon\right\}\\
	\le  L\P\left\{ \|F\|_{C[0,1]}  \le  ( \pi/2 )^{1/4} \varepsilon/\rho\right\}\e^{L\rho/[(1-\rho)\varepsilon]}.
\end{gather*}
Apply \eqref{LiLinde} to see that, as $\varepsilon\downarrow0$,
\[
	-\frac{2\lambda+o(1)}{\pi} \le 
	\varepsilon^4 \log \P\left\{ \| H(\cdot\,,0)\|_{C[0,1]} \le \varepsilon\right\} \le 
	-\frac{2\lambda\rho^4+o(1)}{\pi}.
\]
Since $\rho\in(0\,,1)$ was arbitrary, we may let $\rho$ tend upward to $1$ 
in order to complete the proof of
Proposition \ref{pr:H}.\qed

\subsection{Proof of Proposition \ref{pr:Z}}

We now prove Proposition \ref{pr:Z}. The first step is to establish the analogue of
Proposition \ref{pr:Z} for the more regular process $H$. The following summarizes
that result.

\begin{lemma}\label{lem:small-ball:H}
	For every unbounded, non-increasing, deterministic function $\phi:(0\,,1)\to(0\,,\infty)$,
	\[
		\lim_{\varepsilon\to0^+}
		[\phi(\varepsilon)]^{-1}\log \P\left\{ \| H(\cdot\,,0)\|_{C[0,\varepsilon]} \le 
		\left( \varepsilon / \phi(\varepsilon) \right)^{1/4}
		\right\}=- 2\lambda / \pi.
	\]
\end{lemma}

\begin{proof}
	The random field $H$ inherits scaling properties from 
	white noise and the free-space heat operator. In particular,
	\begin{equation}\label{scale:H}
		\left\{ \rho^{-1/4}H(\rho t\,,\rho^{1/2}x)\, ;\, t\ge0,\, x\in\R\right\}
		\stackrel{d\,}{=} \{ H(t\,,x)\,;\, t\ge0,\ x\in\R\},
	\end{equation}
	for all $\rho>0$. In particular, for every $\varepsilon>0$,
	\[
		\P\left\{\| H(\cdot\,,0)\|_{C[0,\varepsilon]} \le 
		\left(  \varepsilon / \phi(\varepsilon) \right)^{1/4}
		\right\} = \P\left\{ \| H(\cdot\,,0)\|_{C[0,1]} \le
		[\phi(\varepsilon)]^{-1/4} \right\}.
	\]
	The result follows from the above, Proposition \ref{pr:H}, and the fact that $\phi(\varepsilon)\to\infty$
	as $\varepsilon\to0^+$.
\end{proof}

In the next step in the proof of Proposition \ref{pr:Z} we show that $H(t\,,0)$ is very close to
$Z(t\,,0)$. Since $H$ and $Z$ are Gaussian, it suffices to measure closeness using the variance.

\begin{lemma}\label{lem:H-Z}
	$\E( |H(t\,,0) - Z(t\,,0)|^2)\le 5t$ for all $t\ge0$.
\end{lemma}

\begin{proof}
	We can compare \eqref{Z:mild} to \eqref{H:mild} in order to see that
	$\E( |H(t\,,0) - Z(t\,,0)|^2) =2 J_1 + J_2,$ where
	\begin{align*}
		J_1 &= \int_0^t\frac{\d s}{4\pi s}\int_{-1}^1\d y\
			\left| \sum_{n=1}^\infty\exp\left( -\frac{(y+2n)^2}{4s}\right)\right|^2,\\
		J_2 &= \int_0^t\d s\int_{|y|>1}\d y\ [G_{t-s}(y)]^2.
	\end{align*}
	Both terms can be estimated by direct means. Indeed,
	\begin{align*}
		J_1 &\le \int_0^t\frac{\d s}{\pi s}\int_{-1}^1\d y\
			\left| \sum_{n=1}^\infty\exp\left( -\frac{n^2}{4s}\right)\right|^2
			=\frac2\pi\int_0^t 
			\left| \sum_{n=1}^\infty\exp\left( -\frac{n^2}{4s}\right)\right|^2\,\frac{\d s}{s},
			\quad\text{and}\\
		J_2 &=\int_0^t\d s\int_{|y|>1}\d y\ [G_s(y)]^2 =
			\int_0^t\frac{\d s}{\sqrt{8\pi s}}\int_{|z|>\sqrt 2}\d z\ G_s(z),
	\end{align*}
	thanks to the fact that $[G_s(y)]^2 =(4\pi s)^{-1/2} G_s(y\sqrt 2)$ for all $s>0$ 
	and $y\in\R$, and a change of variables.
	Since $\sum_{n=1}^\infty\exp\{ - n^2/(4s) \}\le 
	\int_0^\infty \exp\{ - y^2/(4s)\}\,\d y =  \sqrt{\pi s},$
	it follows that 
	 $J_1\le 2t.$
	And a familiar Gaussian tail bound yields
	\[
		\int_{|z|>\sqrt 2}\d z\ G_s(z)\le \exp\{-1/(2s)\},
	\]
	and hence
	\[
		J_2 \le \int_0^t\exp\left( - \frac{1}{2s}\right)\,\frac{\d s}{\sqrt{8\pi s}}
		\le \frac{t}{\sqrt{8\e\pi}} \le t,
	\]
	thanks to the elementary fact that $s^{-1/2}\exp\{-1/(2s)\}\le \e^{-1/2}$ for all $\theta>0$.
	Combine the bounds for $J_1$ and $J_2$ in order to deduce the lemma. 
\end{proof}

In the next stage of the proof of Proposition \ref{pr:Z} we show that the somewhat crude
approximation offered by Lemma \ref{lem:H-Z} is good enough to yield the closeness of
the respective small-ball probabilities of $t\mapsto H(t\,,0)$ and $t\mapsto Z(t\,,0)$.
In fact, a little extra effort produces the following much better result.

\begin{lemma}\label{lem:H-Z:sup}
	Let $\phi:(0\,,\infty)\to\R_+$ be an unbounded, nonincreasing, deterministic function
	that satisfies the local growth condition \eqref{phi:u}.
	Then,
	\[
		\limsup_{\varepsilon\downarrow0}
		\sqrt{\varepsilon\phi(\varepsilon)}
		\log \P\left\{\|H-Z\|_{C([0,\varepsilon]\times\T)}\ge \left(
		\varepsilon / \phi(\varepsilon) \right)^{1/4}\right\} \le -\frac{1}{10}.
	\]
\end{lemma}

\begin{proof} 
	Let $D(t\,,x) = H(t\,,x)-Z(t\,,x)$ for all $t\ge0$. Clearly, $D$ is a continuous and centered
	Gaussian process with
	\begin{equation}\label{D}\begin{split}
		\E \left( |D(t\,,x)|^2 \right)&\le 5t,\text{ and}\\
		\E \left( |D(t\,,x)-D(s\,,y)|^2 \right)&\lesssim |t-s|^{1/2}+|x-y|,
	\end{split}\end{equation}
	valid uniformly for all $s,t\in[0\,,1]$ and $x,y\in\T$. The first inequality in \eqref{D} 
	is from Lemma \ref{lem:H-Z} and stationarity, and the second is a well-known fact that is
	used frequently in the regularity theory of SPDEs \cite[pp.\ 319--320]{W1986}. Therefore,
	the theory of Gaussian processes (in particular its connection to metric entropy) yields
	a positive number $c$ such that
	\[
		0\le\E \adjustlimits\sup_{s\in[0,\varepsilon]}\sup_{y\in\T} D(s\,,y) \le
		c\left| \varepsilon\log\varepsilon\right|^{1/2}\qquad\text{uniformly
		for all $\varepsilon\in(0\,,1/\e]$};
	\]
	see Ledoux \cite{L1996}. Moreover, by concentration of measure \cite{L1996} and 
	\eqref{D},
	\begin{align*}
		&\P\left\{ \adjustlimits\sup_{s\in[0,\varepsilon]}\sup_{y\in\T}\left| D(s\,,y) - 
            		\E\left[ \adjustlimits\sup_{s\in[0,\varepsilon]}\sup_{y\in\T} 
			D(s\,,y)\right]\right| \ge z\right\}\\
            	&\le 2\exp\left(-\frac{z^2}{2\sup_{s\in[0,\varepsilon]}\Var[D(s\,,0)]}\right)
			\le 2\e^{- z^2/(10\varepsilon)}\quad\text{for all
			$z,\varepsilon>0$}.
	\end{align*}
	Thus we see that, for every $z>0$ and $\varepsilon\in(0\,,1/\e]$,
	\[
		\P\left\{  \|H-Z\|_{C([0,\varepsilon]\times\T)} \ge
		c\left| \varepsilon\log \varepsilon\right|^{1/2} + z\right\} \le 
		2\e^{- z^2/(10\varepsilon)}.
	\]
	This and the moderate-deviations condition \eqref{phi:u} together imply the lemma.
\end{proof}

We have laid the groundwork and are now prepared for the following
conclusion to the results of this section.

\begin{proof}[Proof of Proposition \ref{pr:Z}]
	Choose and fix an arbitrary number $\rho\in(0\,,1)$.
	In accord with Lemmas \ref{lem:small-ball:H} and \ref{lem:H-Z:sup},
	\begin{align*}
		&\P\left\{ \| Z(\cdot\,,0)\|_{C[0,1]} \le \left( 
            		\varepsilon/\phi(\varepsilon) \right)^{1/4}
			\right\} \\
		&\le \P\left\{ \| H(\cdot\,,0)\|_{C[0,\varepsilon]}  \le 
			(1+\rho)\left(  \varepsilon/\phi(\varepsilon) \right)^{1/4}
			\right\}  + 
			\P\left\{ \| H(\cdot\,,0) - Z(\cdot\,,0) \|_{C[0,\varepsilon]}  \ge
			\rho\left( \varepsilon/\phi(\varepsilon) \right)^{1/4}\right\}\\
		&\le \exp\left\{ -\frac{(2\lambda/\pi)+o(1)}{(1+\rho)^4}\phi(\varepsilon)\right\} + 
			\exp\left\{ -\frac{\rho^4+o(1)}{6\sqrt{\varepsilon\phi(\varepsilon)}}\right\}
			\qquad\text{as $\varepsilon\downarrow0$}.
	\end{align*}
	Since $\rho\in(0\,,1)$ can be as close to zero as we want, this and \eqref{phi:u} together imply that
	\[
		\limsup_{\varepsilon\downarrow0} [\phi(\varepsilon)]^{-1}\log
		\P\left\{ \| Z(\cdot\,,0)\|_{C[0,\varepsilon]} \le\left( 
		\varepsilon/\phi(\varepsilon) \right)^{1/4}
		\right\} \le - 2\lambda/\pi.
	\]
	Likewise, we appeal to Lemmas \ref{lem:small-ball:H} and \ref{lem:H-Z:sup} as follows:
	\begin{align*}
		&\exp\left\{ -\frac{(2\lambda/\pi)+o(1)}{(1-\rho)^4}\phi(\varepsilon)\right\}
			=\P\left\{ \| H(\cdot\,,0)\|_{C[0,\varepsilon]} 
			\le (1-\rho)\left( \varepsilon/\phi(\varepsilon) \right)^{1/4}
			\right\}\\
		&\le \P\left\{ \| Z(\cdot\,,0)\|_{C[0,\varepsilon]} \le 
			\left( \varepsilon / \phi(\varepsilon) \right)^{1/4}
			\right\} + 
			\P\left\{ \| H(\cdot\,,0)- Z(t\,,0)\|_{C[0,\varepsilon]} \ge
			\rho\left( \varepsilon / \phi(\varepsilon) \right)^{1/4}\right\}\\
		&\le\P\left\{ \| Z(\cdot\,,0)\|_{C[0,\varepsilon]}  \le 
			\left( \varepsilon / \phi(\varepsilon) \right)^{1/4}
			\right\} + 
			\exp\left\{ -\frac{\rho^4+o(1)}{6}
			\sqrt{\frac{\phi(\varepsilon)}{\varepsilon}}\right\}
			\qquad\text{as $\varepsilon\downarrow0$}.
	\end{align*}
	Since $\rho\in(0\,,1)$ can be as close to zero as we want, this and \eqref{phi:u} together imply that
	$\liminf_{\varepsilon\downarrow0} [\phi(\varepsilon)]^{-1}\log
	\P\{ \| Z(\cdot\,,0)\|_{C[0,\varepsilon]} \le
	( \varepsilon / \phi(\varepsilon) )^{1/4}
	\} \ge - 2\lambda/\pi,$
	and concludes the proof of the proposition.
\end{proof}

\section{Linearization, and proof of Theorem \ref{th:u}} \label{sec:linearization} 

Consider the space-time random field $\mathscr{E}$ that is defined 
by setting, for all $t\ge0$ and $x\in\T$,
\[
	\mathscr{E}(t\,,x) = u(t\,,x) - (p_t*u_0)(x) - \sigma(u_0(x))Z(t\,,x).
\]
Thus, the random variable $\mathscr{E}(t\,,x)$ measures the linearization error of the solution
to \eqref{u} at the space-time point $(t\,,x)\in\R_+\times\T$. It is  known that 
$\mathscr{E}(t\,,x)\approx 0$ when $t\approx 0$; this was done independently and
nearly at the same time in \cite{KSXZ2013} and 
\cite{HP2015}. The method of \cite{KSXZ2013} provided detailed bounds for 
the moments of $\sup|\mathscr{E}|$ but with suboptimal $t$-dependent rates,
and the method of \cite{HP2015} provided a.s.\ estimates for $\sup|\mathscr{E}|$,
with nearly sharp control of the size of $\sup|\mathscr{E}|$, but only
under extra smoothness conditions on $\sigma$; specifically $\sigma$ was assumed
to be in $C^r$ for a large enough $r\ge 3$. Our next proposition 
improves both of these results. It yields a rate that is unimproveable to leading
order, does not require additional smoothness for $\sigma$, and provides quantitative
bounds on $\P\{\mathscr{E}\approx0\}$. More precisely, we have

\begin{proposition}\label{pr:localize}
	If $\sigma$ is bounded, then for every $\nu>0$
	there exists $a=a(\nu)\in(0\,,1)$ such that
	\[
		\P\left\{ \|\mathscr{E}\|_{C([0,t]\times\T)} \ge  a t^{1/2}\, \log_+(1/t) \right\}
		\lesssim t^\nu\text{ uniformly for every $t\in(0\,,1)$.}
	\]
\end{proposition}

The preceding appears to be a result that is useful only when $t\approx0$.
However, it is possible to combine it with the Markov property of the solution
to improve itself. We will not delve into this topic since we do not need
it at present.

The proof of Proposition \ref{pr:localize} requires a few preliminary calculations.
Before we commence with those, let us quickly deduce the following analogue of
a result in \cite{HP2015} but valid with no additional smoothness assumptions
on $\sigma$ and with a slightly tighter error rate at the sharp leading order of $t^{1/2}$. 
We will not need Corollary \ref{cor:localize} in the sequel
and mention it only to record the fact for potential later use elsewhere.

\begin{corollary}\label{cor:localize}
	Regardless of whether or not $\sigma$ is bounded, there exists an a.s.-finite random variable 
	$V$ such that $\|\mathscr{E}(t)\|_{C(\T)} \le Vt^{1/2}\,\log_+(1/t)$ 
	uniformly for all $t\in[0\,,1]$.
\end{corollary}

\begin{proof}
	The proof uses a stopping-time argument.
	Choose and fix a real number $N>0$,
	and define $u_N$ the same as $u$ -- see \eqref{u} -- but with $\sigma$ replaced by
	$\sigma_N$
	\[
		\sigma_N(x) = \begin{cases}
			\sigma(N)&\text{if $x>N$},\\
			\sigma(x) &\text{if $-N<\sigma(x)\le N$},\\
			\sigma(-N)&\text{if $x\le -N$}.
		\end{cases}
	\]
	That is, $u_N(0)=u_0$, and
	\[
		u_N(t\,,x) = (p_t*u_0)(x) + \int_{(0,t)\times\R} p_{t-s}(x\,,y)
		\sigma_N(u_N(s\,,y))\, W(\d s\,\d y),
	\]
	for $t>0$ and $x\in\T$. Define
	\[
		T_N = \inf\left\{ t\ge0:\, \|u_N(t)\|_{C(\T)} >N \right\}\qquad
		[\inf\varnothing=\infty].
	\]
	Then, $T_N$ is a stopping time with respect to the filtration $\F$ generated by the noise.
	Basic properties of the Walsh stochastic integral and the continuity of $u$ and $u_N$ together
	imply that
	\begin{equation}\label{uu_N}
		\P\left\{ u_N(t)=u(t)\text{ for all $t < T_N$}\right\}=1,
	\end{equation}
	whence also $T_N = \inf\{ t\ge0:\, \| u(t)\|_{C(\T)} >N \}$ almost surely.
	Therefore, we apply Proposition \ref{pr:localize} with $\nu=1$ in order
	to see that there exists $a>0$ such that
	\begin{align*}
		&\P\left\{ \sup_{t\in[0,\varepsilon]}\|\mathscr{E}(t)\|_{C(\T)}\ge
			a \varepsilon^{1/2}\vert\!\log\varepsilon\vert\,; T_N>1\right\}\\
		&=\P\left\{ \sup_{t\in[0,\varepsilon]}
			\left\| u_N(t) - p_t*u_0 - \sigma(u_0)Z(t)\right\|_{C(\T)} \ge
			a \varepsilon^{1/2}\vert\!\log\varepsilon\vert\,;T_N>1\right\}\\
		&\le\P\left\{ \sup_{t\in[0,\varepsilon]}
			\left\| u_N(t) - p_t*u_0 - \sigma(u_0)Z(t)\right\|_{C(\T)} \ge
			a \varepsilon^{1/2}\vert\!\log\varepsilon\vert\right\}\lesssim\varepsilon,
	\end{align*}
	uniformly for all $\varepsilon\in(0\,,1)$. Replace $\varepsilon$ by $\exp(-n)$ as $n$ ranges over
	$\N$ and sum over $n$ to deduce from the Borel-Cantelli lemma that
	\begin{equation}\label{leqa}
		\limsup_{n\to\infty} \sup_{t\in[0,\exp(-n)]} \frac{%
		\|\mathscr{E}(t)\|_{C(\T)} }{\e^{-n/2}n}\le a
		\ \text{ a.s.\ on $\{T_N>1\}$}.
	\end{equation}
	If $\exp(-n-1)\le \varepsilon\le \exp(-n)$ and $n\in\N$, then
	\[
		\sup_{s\in[0,\varepsilon]}
		\frac{\|\mathscr{E}(s)\|_{C(\T)}}{\varepsilon^{1/2}\vert\!\log\varepsilon\vert}\le
		\sup_{t\in[0,\exp(-n)]}\frac{%
		\|\mathscr{E}(t)\|_{C(\T)}}{\e^{-(n-1)/2}n}.
	\]
	Therefore, \eqref{leqa} implies that
	\[
		\P\left\{\limsup_{\varepsilon\downarrow0}
		\sup_{s\in[0,\varepsilon]}
		\frac{\|\mathscr{E}(s)\|_{C(\T)}}{\varepsilon^{1/2}\vert\!\log\varepsilon\vert}
		\le \frac{a}{\e^{1/2}}\right\}\ge \P\{T_N>1\}
		\qquad\text{for all $N>0$}.
	\]
	Because \eqref{uu_N} and the a.s.-continuity 
	of $u$ together imply that $\lim_{N\to\infty}T_N=\infty$ a.s., this proves that
	\[
		\limsup_{\varepsilon\downarrow0}\sup_{s\in[0,\varepsilon]}
		\frac{\|\mathscr{E}(s)\|_{C(\T)}}{\varepsilon^{1/2}\vert\!\log\varepsilon\vert}
		\le \frac{a}{\sqrt{e}} \qquad\text{a.s.}
	\]
	In particular, the above limsup is finite almost surely. This is another way to state the corollary.
\end{proof}
Now we begin proof of Proposition \ref{pr:localize} in earnest.
Let us define a metric $\Delta$ on space-time $\R_+\times\T$ by setting
\[
	\Delta\left( (t\,,x)\,,(s\,,y) \right) =  |t-s|^{1/4} + |x-y|^{1/2}
	\quad\text{for all $s,t\ge0$ and $x,y\in\T$}.
\]
It might help to recall that we are using the additive notation for elements of $\T$.
In particular, $|x-y|^{1/2}$ is shorthand for $|x-y\pmod 2|^{1/2}$ whenever $x,y\in\T$.

The following is a consequence of the large-deviations
result of Sowers \cite{S1992} and well-known relations between 
tails of a Gaussian law and its moments. Results of the following type
are well known and typically used to prove that the process $u$
is continuous all the way up to and including the boundary of $[-1\,,1]$,
keeping in mind also that $\pm1$ are identified with one another here.

\begin{lemma}[Sowers \cite{S1992}]\label{lem:modulus:u}
	If $\sigma$ is bounded, then 
	\[
		\| u(t\,,x) - u(s\,,y) \|_k \lesssim \sqrt k\, \Delta\big((t\,,x)\,,(s\,,y)\big),
	\]
	uniformly for all $x,y\in\T$, $s,t\ge0$, and $k\ge 2$.
\end{lemma}

Next, we present an exponential tail estimate for the linearization error $\mathscr{E}$,
valid when $\sigma$ is bounded.

\begin{lemma}\label{lem:localize}
	If $\sigma$ is bounded, then there exists $\gamma>0$ such that
	\[
		\adjustlimits\sup_{t\in(0,1]}\sup_{x\in\T}
		\E\exp\left( \gamma\left|\frac{ \mathscr{E}(t\,,x) }{\sqrt t}
		\right|\right)
		<\infty.
	\]
\end{lemma}

\begin{proof}

	Compare \eqref{u:mild} with \eqref{Z:mild} in order to see that
	\[
		\mathscr{E}(t\,,x) = 
		\int_{(0,t)\times\T} p_{t-s}(x\,,y)\left[ \sigma(u(s\,,y)) - \sigma(u(0\,,x))\right]\,
		W(\d s\,\d y),
	\]
	for all $t>0$ and $x\in\T$. Thanks to the  Young's inequality for stochastic convolutions
	(see Khoshnevisan \cite[Proposition 5.2]{K2014}),
	we have the following for every real number $k\ge 2$, $t>0$, and $x\in\T$:
	\begin{equation}\label{scrE}\begin{split}
		\left\| \mathscr{E}(t\,,x)\right\|_k^2
		&\le 4k\int_0^t\d s\int_\T\d y\ [p_{t-s}(x\,,y)]^2
			\left\| \sigma(u(s\,,y)) - \sigma(u(0\,,x))\right\|_k^2\\
		&\le 4k[\lip(\sigma)]^2\int_0^t\d s\int_\T\d y\ [p_{t-s}(x\,,y)]^2
			\left\| u(s\,,y)  - u(0\,,x)\right\|_k^2\\
		&\lesssim k^2\int_0^t\d s\int_\T\d y\ \left[ p_{t-s}(x\,,y)
			\Delta((0\,,x)\,,(s\,,y))\right]^2.
	\end{split}\end{equation}
	Thanks to \eqref{p:G},  $(t\,,a)\mapsto p_r(a) - G_r(a)$  is bounded 
	uniformly on $\R_+\times\T$. In this way, we find that
	\[
		\left\| \mathscr{E}(t\,,x)\right\|_k^2
		\lesssim k^2\int_0^t\d s\int_{-\infty}^\infty\d y\ [G_s(y)]^2
		\left(\sqrt{t-s} + |y| \right)
		+ k^2\int_0^t\d s\int_\T\d y
		\left( \sqrt s + |y| \right),
	\]
	uniformly for all $t>0$, $x\in\T$, and $k\ge2$.
	Direct computation yields the bound,
	\[
		\int_0^t\d s\int_\T\d y \left( \sqrt s + |y| \right) \propto t^{3/2} +t,
		\qquad\text{valid uniformly for all $t>0$.}
	\]
	Similarly, we find that
	for all $t>0$,
	\begin{align*}
		\int_0^t\d s\int_{-\infty}^\infty\d y\ [G_s(y)]^2 
			\sqrt{t-s} &=\int_0^t\sqrt{t-s}\, G_{2s}(0)\,\d s
			\propto t,\quad\text{and}\\
		\int_0^t\d s\int_{-\infty}^\infty\d y\ [G_s(y)]^2|y| &= 
			\int_0^t\frac{\d s}{s}\int_{-\infty}^\infty\d y\ [G_1(y/\sqrt s)]^2|y|\\
		&= \int_0^t\d s\int_{-\infty}^\infty\d w\ [G_1(w)]^2|w|\propto t,
	\end{align*}
	where the constants of proportionality do not depend on $t$. 
	It follows from the preceding effort that there exists $B>0$ such that
	\[
		\adjustlimits\sup_{t\in(0,1]}\sup_{x\in\T}
		\E\left(| \mathscr{E}(t\,,x) |^k\right) \le (Bk)^k t^{k/2}
		\quad\text{uniformly for all $k\ge2$.}
	\]
	By Jensen's inequality, the preceding in fact
	holds uniformly for all $t>0$, $x\in\T$, and $k\ge1$. 
	Among other things, this and
	Stirling's formula together yield a constant $C>0$ such that
	\[
		\adjustlimits\sup_{t\in(0,1]}\sup_{x\in\T}
		\E\left(\left| \frac{\mathscr{E}(t\,,x)}{\sqrt t} \right|^k\right)
		\le C^k k!
		\quad\text{uniformly for all  $k\in\Z_+$.}
	\]
	Choose and fix
	an arbitrary $\gamma\in(0\,,C)$ and sum the above inequality
	over all $k\in\Z_+$ in order to deduce the lemma.
\end{proof}

\begin{remark}
	We can make an adjustment to the preceding proof in order to
	see that the distribution of $\mathscr{E}(t\,,x)$
	in fact has Gaussian tails when $\sigma$ is bounded. 
	However, in order to achieve a Gaussian tail that is valid
	uniformly in $t\in(0\,,1]$ all the way down to $t=0$, we
	need to normalize $\mathscr{E}(t\,,x)$ differently. A more precise 
	statement is this: \emph{There exists $\gamma'>0$ such that}
	\begin{equation}\label{ADJ}
		\adjustlimits\sup_{t\in(0,1]}\sup_{x\in\T}
		\E\exp\left( \gamma'\left|\frac{ \mathscr{E}(t\,,x) }{t^{1/4}}
		\right|^2\right)
		<\infty.
	\end{equation}
	To prove \eqref{ADJ} we simply adjust the first line of \eqref{scrE} by bounding out
	the difference of the $\sigma$'s. In this way we obtain the following, thanks to
	the semigroup property of the heat kernel and a standard bound on the heat kernel
	on $\T$ at small times: Uniformly for all $t\in(0\,,1]$, $x\in\T$, and $k\ge 2$,
	\begin{align*}
		\left\| \mathscr{E}(t\,,x)\right\|_k^2
			&\le 16\sup_{z\in\R}|\sigma(z)|^2
			k\int_0^t\d s\int_\T\d y\ [p_{t-s}(x\,,y)]^2 \\
		&=  16\sup_{z\in\R}|\sigma(z)|^2k\int_0^t p_{2s}(0\,,0)\,\d s
			\lesssim k\sqrt t.
	\end{align*}
	This inequality yields \eqref{ADJ}. By itself, the rate $t^{1/4}$ renders the bound
	\eqref{ADJ} useless since the individual terms that define $\mathscr{E}$
	are each of the order $t^{1/4}$ in law when $t\approx0$. However, the observation
	has its uses. For example, \eqref{ADJ} is good enough to ensure
	that, among other things, $\mathscr{E}(t\,,x)$ has Gaussian
	probability tails. 
\end{remark}

The preceding remark can be followed up by our next lemma which
describes unconditional Gaussian tails for the
distribution of the spatio-temporal increments of $\mathscr{E}$
when $\sigma$ is bounded.

\begin{lemma}\label{lem:modulus:Delta}
	If $\sigma$ is bounded, then
	there exists a number $\gamma_0>0$ such that
	\[
		\E\exp\left(  \sup_{0<s<t\le 1}\sup_{\substack{x,y\in\T\\x\neq y}}
		\gamma_0\left|\frac{\mathscr{E}(t\,,x)-\mathscr{E}(s\,,y)}{%
		\Delta((t\,,x)\,,(s\,,y))\sqrt{ \log_+(1/ \Delta((t\,,x)\,,(s\,,y)))}}\right|^2\right)<\infty.
	\]
\end{lemma}

\begin{proof}
	Since the random field $Z$ is defined in the same way as the random field $u$ but
	with $\sigma\equiv1$, Lemma \ref{lem:modulus:u} implies that
	\[
		\| Z(t\,,x) - Z(s\,,y) \|_k \lesssim \sqrt{k}\,\Delta((t\,,x)\,,(s\,,y)),
	\]
	uniformly for all $x,y\in\T$, $s,t\ge0$, and $k\ge 2$. Therefore, the boundedness
	and Lipschitz continuity of $\sigma$ yield the bounds,
	\begin{align*}
		&\| \sigma(u_0(x))Z(t\,,x) - \sigma(u_0(y))Z(s\,,y)\|_k\\
		&\le \| \sigma(u_0(x))Z(t\,,x) - \sigma(u_0(x))Z(s\,,y)\|_k
			+ |\sigma(u_0(x)) - \sigma(u_0(y))|\|Z(s\,,y)\|_k\\
		&\lesssim \sqrt k\,\Delta((t\,,x)\,,(s\,,y)) + |x-y|\|Z(s\,,y)\|_k,
	\end{align*}
	valid  uniformly for all $k\ge 2$, $s>0$, and $y\in\T$. 
	Since $Z$ is a Gaussian random field, a standard computation yields
	\[
		\|Z(s\,,y)\|_k\lesssim\sqrt{k}\,\|Z(s\,,y)\|_2\lesssim\sqrt{k}\, s^{1/4},
	\]
	uniformly for all $k\ge 2$, $s>0$, and $y\in\T$. It follows that
	\[
		\| \sigma(u_0(x))Z(t\,,x) - \sigma(u_0(y))Z(s\,,y)\|_k
		\le  \sqrt k\Delta((t\,,x)\,,(s\,,y)),
	\]
	uniformly for all $k\ge 2$, $s,t\in(0\,,1]$, and $x,y\in\T$.
	It is well known that, because $u_0$ is Lipschitz  continuous,
	\[
		\left| (p_t*u_0)(x) - (p_s*u_0)(y) \right| \lesssim \Delta((t\,,x)\,,(s\,,y)),
	\]
	uniformly for all $s,t\in(0\,,1]$ and $x,y\in\T$.\footnote{
	This follows for example, from the fact that we can write $(p_t*u_0)(x)=\E u_0(x+B_t)$
	for a Brownian motion
	$\{B_t\}_{t\ge0}$ on $\T$ with speed 2, so that
	$| (p_t*u_0)(x) - (p_s*u_0)(y) |\le 
    	\| u_0(x+B_t) - u_0(y+B_s)\|_1
	\le \|u_0\|_{C^{1/2}(\R)} \{ \|B_t-B_s\|_1 + |x-y|\}^{1/2}$, by the triangle inequality.
	}
	Therefore, the preceding
	bounds together yield the inequality
	\[
		\|\mathscr{E}(t\,,x)-\mathscr{E}(s\,,y)\|_k \lesssim \sqrt k\,
		\Delta((t\,,x)\,,(s\,,y)),
	\]
	valid uniformly for all $s,t\in(0\,,1]$ and $x,y\in\T$. Now a standard metric
	entropy argument completes the proof.
\end{proof}

We are ready to establish the following, which is a slightly weaker fixed-time
version of Proposition 
\ref{pr:localize}, and paves the way toward proving afterward
that proposition in complete generality.

\begin{lemma}\label{lem:max_x(E)}
	If $\sigma$ is bounded, then 
	for every $\nu>0$, there exists a number $K=K(\nu)>1$ such that
	\[
		\P\left\{ \|\mathscr{E}(t)\|_{C(\T)} \ge  K \sqrt t\,\log_+(1/t) \right\}
		\lesssim t^\nu \text{ uniformly for all $t\in(0\,,1)$.}
	\]
\end{lemma}

\begin{proof}
	Lemmas \ref{lem:localize} and \ref{lem:modulus:Delta},
	and Chebyshev's inequality together
	yield a number $C>0$ such that
	\begin{equation}\label{bup}\begin{split}
		\adjustlimits\sup_{t>0}\sup_{x\in\T}\P\left\{ |\mathscr{E}(t\,,x)|
			\ge 2\beta \sqrt t \right\}
			&\lesssim\exp\left(-C\beta\right),\quad\text{and}\\
		\P\left\{ \sup_{\substack{x,y\in\T:\\|x-y|\le\varepsilon}}
			|\mathscr{E}(t\,,x) - \mathscr{E}(t\,,y)| \ge \theta\sqrt{\varepsilon\log(1/\varepsilon)}
			\right\} &\lesssim\exp\left( - C\theta^2\right),
	\end{split}\end{equation}
	uniformly for every $\beta,\theta>0$. Define
	\[
		\T_n = \cup_{i\in[-n,n-1]\cap\Z}\left\{ i/n\right\} \qquad\text{for all $n\in\N$},
	\]
	and remember that because of the group topology of the torus,
	the ends of $\T=[-1\,,1]$ are identified with one another.
	This shows that every point in $\T$ is within $n^{-1}$ of some point in $\T_n$. 
	Because the cardinality of $\T_n$ is $\le 4n$ uniformly for all $n\in\N$, we can
	deduce from \eqref{bup} that
	\begin{align*}
		&\P\left\{ \|\mathscr{E}(t)\|_{C(\T)} \ge \frac{2\beta}{C} \sqrt t \right\}\\
		&\le \P\left\{ \max_{x\in\T_n}|\mathscr{E}(t\,,x)| \ge \frac{\beta}{C} \sqrt t \right\}
			+ \P\left\{ \sup_{\substack{x,y\in\T:\\|x-y|\le1/n}}
			|\mathscr{E}(t\,,x)-\mathscr{E}(t\,,y)| \ge\frac{\beta}{C}\sqrt t\right\}\\
		&\lesssim n\e^{- \beta/2 } +
			\exp\left( -\frac{ \beta^2 tn}{C\, \log_+(n) }\right),
	\end{align*}
	uniformly for all $t\in(0\,,1)$, $\beta>0$, and $n\in\N$. 
	We apply the preceding with
	\[
		\beta = \beta(t) = 2 \kappa\log_+(1/t)
		\quad\text{and}\quad
		n = n(t) = C\left\lfloor 1/t\right\rfloor,
	\]
	where $\kappa>\nu\vee 1$ is a fixed number, in order to deduce the result.
\end{proof}

We are ready to establish Proposition \ref{pr:localize}.

\begin{proof}[Proof of Proposition \ref{pr:localize}]
	Lemma \ref{lem:modulus:Delta} implies that there exists $C>0$ such that
	\begin{equation}\label{gub}
		\P\left\{ \sup_{\substack{s,r\in(0,1]:\\|s-r|\le\varepsilon}}
		\|\mathscr{E}(s) - \mathscr{E}(r) \|_{C(\T)} \ge \theta \, 
		\varepsilon^{1/4} \left[\log(1/\varepsilon)\right]^{1/2} 
		\right\} \lesssim\e^{- C\theta^2},
	\end{equation}
	uniformly for all $\varepsilon\in(0\,,1)$ and $\theta>0$.
	Now, let us choose and fix some $t\in(0\,,1)$,  define
	\[
		S_{n,t} = \cup_{j\in[n-1,n]\cap\N}\left\{ jt/n\right\} \qquad\text{for all $n\in\N$},
	\]
	and observe that every point in $[0\,,t]$ is certainly within $1/n$ of some point in $S_{n,t}$. 
	Because the cardinality of $S_{n,t}$ is $\lesssim n$ uniformly for all $n\in\N$, we can
	deduce from Lemma \ref{lem:max_x(E)} and eq.\ \eqref{gub} that for every $\nu>0$ there
	exists $K=K(\nu)>0$ such that
	\begin{align*}
		&\P\left\{ 
			\| \mathscr{E}\|_{C([0,t]\times\T)} \ge  2K\sqrt{t}\,\log_+(1/t)\right\}
			\le \P\left\{ 
			\| \mathscr{E}\|_{C(S_{n,t}\times\T)} \ge   K\sqrt{t}\,\log_+(1/t)\right\}\\
		&\hskip1in + \P\left\{ \textstyle\sup_{\substack{s,r\le t:\\|s-r|\le1/n}}
			\|\mathscr{E}(r)-\mathscr{E}(s)\|_{C(\T)} \ge K\sqrt{t}\,\log_+(1/t)\right\}\\
		&\lesssim n t^{\nu+2}+
			\exp\left( -CK^2\frac{\sqrt{n}}{\log_+ (n)}\, t|\log_+(1/t)|^2\right),
	\end{align*}
	uniformly for all $n\in\N$ and $t\in(0\,,1)$. Apply the preceding with 
	$n=\lfloor t^{-2}\rfloor$ to complete the proof.
\end{proof}

Proposition \ref{pr:localize} forms
the bulk of the effort of proving Theorem \ref{th:u}.
Now that we have proved the proposition, we can 
conclude the proof of Theorem \ref{th:u},
which is the first primary offering of this work.

\begin{proof}[Proof of Theorem \ref{th:u}]
 We can observe that
	\begin{align*}
		&\P\left\{ \textstyle\sup_{t\in[0,\varepsilon]} | u(t\,,x)  - (p_t*u_0)(x)| \le 
			\left( \varepsilon / \phi(\varepsilon) \right)^{1/4}\right\}\\
		&\le \P\left\{ \textstyle\sup_{t\in[0,\varepsilon]} |\sigma(u_0(x))Z(t\,,x)|\le 
			\left( \varepsilon / \phi(\varepsilon) \right)^{1/4} + a\varepsilon^{1/2}
			|\log\varepsilon|\right\} \\
		&\hskip2in +  
			\P\left\{ \textstyle\sup_{t\in[0,\varepsilon]}|\mathscr{E}(t\,,x)|\ge
			a\varepsilon^{1/2}|\log\varepsilon|\right\}.
	\end{align*}
	Thanks to \eqref{phi:u}, the first probability on the right-hand side decays at least as 
	rapidly as $\exp\{-(2\lambda\sigma^4(u_0(x))/\pi)+o(1))\phi(\varepsilon)\}$.
	Therefore, we choose $\nu>2\lambda |\sigma(u_0(x))| /\pi$ too see that, as long as we pick
	$a$ large enough (which we will), Lemma \ref{lem:max_x(E)} assures us that,
	as $\varepsilon\downarrow0$,
	\begin{align*}
		&\P\left\{ \textstyle\sup_{t\in[0,\varepsilon]} | u(t\,,x)  - (p_t*u_0)(x)| \le 
			\left( \varepsilon / \phi(\varepsilon) \right)^{1/4}\right\}\\
		&\lesssim\exp\left\{ -\left(\frac{2\lambda [\sigma(u_0(x))]^4 +o(1)}{\pi}\right)
			\phi(\varepsilon)\right\}+ 
			\varepsilon^\nu
			\le \exp\left\{ -\left(\frac{2\lambda [\sigma(u_0(x))]^4 +o(1)}{\pi}\right)
			\phi(\varepsilon)\right\},
	\end{align*}
	see \eqref{phi:u}.
	In order to derive a complementary bound, we  write
	\begin{align*}
		&\P\left\{ \textstyle\sup_{t\in[0,\varepsilon]} |\sigma(u_0(x))Z(t\,,x)|\le 
			\left( \varepsilon / \phi(\varepsilon) \right)^{1/4} +a\varepsilon^{1/2}
			|\log\varepsilon| \right\}\\
		&\le \P\left\{ \textstyle\sup_{t\in[0,\varepsilon]} | u(t\,,x)  - (p_t*u_0)(x)| \le 
			\left( \varepsilon / \phi(\varepsilon) \right)^{1/4} \right\}\\
		&\hskip2in+  \P\left\{ \textstyle\sup_{t\in[0,\varepsilon]}|\mathscr{E}(t\,,x)|\ge
			a\varepsilon^{1/2}|\log\varepsilon| \right\},
	\end{align*}
	and proceed in parallel to the previous part. This completes the proof
	since \eqref{phi:u} assures that
	\[
		\sup_{t\in[0,\varepsilon]}
		\left\| p_t*u_0 - u_0 \right\|_{C(\T)}
		\le\lip(u_0)\sqrt{\varepsilon} = o\left( (\varepsilon/\phi(\varepsilon))^{1/4}\right),
	\]
	as $\varepsilon\downarrow0$. This completes the proof.
\end{proof}

\section{Proof of Corollary \ref{cor:Chung}} \label{sec:proof:cor}

As was mentioned in the Introduction, one might anticipate
some version of Corollary \ref{cor:Chung}, viewed as a natural
byproduct of Theorem \ref{th:u}. However, it turns out that
the proof of Corollary \ref{cor:Chung} requires the introduction
of a few subtle ideas that are not altogether standard. Therefore, we use this section to hash out
the details of that argument.
Throughout this section, let us define
\[
	\psi(t) = \left(\frac{t}{\log|\log_+(1/t)|}\right)^{1/4}\qquad\text{for all $t\ge0$},
\]
and recall the Gaussian random field $H$ from \eqref{H} and \eqref{H:mild}.
The following is the main step of the proof of Corollary \ref{cor:Chung}.

\begin{proposition}\label{pr:Chung:H}
	For every $x\in\R$,
	\[
		\liminf_{\varepsilon\downarrow0} \sup_{t\in[0,\varepsilon]}\frac{|H(t\,,x)|}{\psi(\varepsilon)}
		 = \left(\frac{2\lambda}{\pi}\right)^{1/4} \qquad
		\text{a.s.}
	\]
\end{proposition}

Before we prove Proposition \ref{pr:Chung:H}, we pause to quickly verify Corollary \ref{cor:Chung}.
Then, we concentrate on proving Proposition \ref{pr:Chung:H}, which is the main portion of the work.

\begin{proof}[Sketch of a conditional proof of Corollary \ref{cor:Chung} given Proposition \ref{pr:Chung:H}]
	We may apply Lemma \ref{lem:H-Z:sup} [with $\phi(t)=\delta^{-4}\log\log_+(1/t)$]
	to see that for every $\delta>0$ there exists $K=K(\delta)>0$ such that
	\begin{equation}\label{P(H-Z)}
		\P\left\{ \sup_{t\in[0,\varepsilon]}|H(t\,,x)-Z(t\,,x)| \ge \delta\psi(\varepsilon)
		\right\}\le 
		K\exp\left( -\frac{1}{\sqrt{K\varepsilon\log|\log\varepsilon|}}\right),
	\end{equation}
	uniformly for all $\varepsilon\in(0\,,\e^{-4})$. 
	Because $\delta>0$
	is arbitrary, the Borel-Cantelli lemma then implies that, with probability one,
	\[
		\sup_{t\in[0,\varepsilon]}\|H(t)-Z(t) \|_{C(\T)}
		=o(\psi(\varepsilon))\qquad\text{as $\varepsilon\downarrow0$}.
	\]
	Because $\varepsilon^{1/2}|\log\varepsilon|\ll \psi(\varepsilon)$ as $\varepsilon\downarrow0$,
	the preceding and Corollary \ref{cor:localize} together yield Corollary \ref{cor:Chung}. We leave the
	remaining details to the interested reader.
\end{proof}

Now we start to prove Proposition \ref{pr:Chung:H}.
From here on, let us choose a fixed real number $\alpha>0$,
and define
\begin{equation}\label{t_n}
	t_n  = \exp\left( - n^{1+\alpha}\right)\qquad\text{for every $n\in\N$}.
\end{equation}
Because $\alpha>0$, a Taylor expansion yields 
\begin{equation}\label{t/t}
	 \frac{t_{n+1}}{t_n}\le \exp\left( -(1+\alpha)n^\alpha \right)\qquad
	\text{uniformly for all $n\in\N$}.
\end{equation}

\begin{lemma}\label{lem:P(H:n+1)}
	For every $\delta\in(0\,,1)$ there exists $M=M(\delta\,,\alpha)$
	such that, uniformly for all $n\in\N$,
	\[
		\P\left\{ \|H\|_{C([0,t_{n+1} ]\times[-1,1]) } 
		\ge \delta\psi(t_n)
		\right\} \le M \exp\left(-  \frac{ \exp\left( (1+\alpha)n^\alpha/2\right)}{M
		\sqrt{\log_+(n)}} \right).
	\]
\end{lemma}

\begin{proof}
	It is not hard to see that $\Var(H(t\,,0))\propto \sqrt t$ uniformly for all $t>0$; this is very well known, but
	also follows essentially immediately from the scaling property
	\eqref{scale:H} of the random field $H$. It is also well known that, for every fixed $T>0$,
	$\|H(t\,,x)-H(s\,,y)\|_2\asymp |t-s|^{1/4}+|x-y|^{1/2}$ uniformly for all $x,y\in\R$ and $s,t\in[0\,,T]$.
	In fact, Lemma \ref{lem:modulus:u} asserts this in a more general context where $\sigma$ can be
	nonlinear. A suitable version of Dudley's metric entropy theorem
	\cite[Theorem 6.1]{L1996} yields a constant $L>0$
	such that
	\[
		\E\left( \|H\|_{C([0,t]\times[-1,1]) } \right)\le L t^{1/4}\sqrt{\log_+(1/t)}
		\qquad\text{for all $t\in(0\,,1)$}.
	\]
	Now we may apply concentration of measure \cite{L1996} in order to see that there exists $\ell>0$
	such that
	\begin{equation}\label{prep}
		\P\left\{ \|H\|_{C([0,t_{n+1} ]\times[-1,1]) }
		\ge Lt_{n+1}^{1/4}|\log t_{n+1}|^{1/2} + z
		\right\} \le2\exp\left(- \frac{\ell z^2}{\sqrt{t_{n+1}}}\right),
	\end{equation}
	for all $n\in\N\cap[2\,,\infty)$ and $z>0$. Thanks to \eqref{t/t},
	\begin{align*}
		&t_{n+1}^{1/4} |\log t_{n+1}|^{1/2} = \left( \frac{t_{n+1}}{t_n}\right)^{1/4}
			 |\log t_{n+1}|^{1/2}\left( \log|\log t_n|\right)^{1/4}
			\psi(t_n)\\
		&\le\exp\left( -\frac{(1+\alpha)n^\alpha}{4}\right)
			(n+1)^{1/2}(\log n)^{1/4}\psi(t_n)
			\le \frac{\delta}{2L} \psi(t_n),
	\end{align*}
	uniformly for all $n$ large enough, and how large depends only on $(\delta\,,\alpha)$.
	Therefore, we plug into \eqref{prep} 
	$z = \delta\psi(t_n)/2$,
	and deduce the asserted inequality of the lemma for large $n$
	after a few lines of computation. We may increase the constant $M$, if it is needed,
	in order to obtain
	the lemma for all $n\in\N$.
\end{proof}

Next we adopt a localization idea of Lee and Xiao \cite{LX2023}, and define
a family $\{H_n\}_{n\in\N}$ of space-time Gaussian random fields 
by setting
\begin{equation}\label{Hn}
	H_n(t\,,x) = \int_{[t_{n+1},t)\times\R} G_{t-s}(y-x)\,W(\d s\,\d y),
\end{equation}
for all $(t\,,x)\in[t_{n+1}\,,t_n]\times\R$.
If $n\gg1$ then $H_n\approx H$. The following is a careful way to say this,
and contains also a tight quantitative bound on the approximation error, necessary
for small-ball probability estimates that follow.

\begin{lemma}\label{lem:Var(H-H)}
	For every $\delta\in(0\,,1)$ there exists $M=M(\delta\,,\alpha)>0$ such that,
	uniformly for all $n\in\N$,
	\[
		\P\left\{ \| H - H_n \|_{C([t_{n+1},t_n]\times[-1,1])} \ge \delta\psi(t_n)\right\}
		\le M\exp\left( -\frac{\exp\left( (1+\alpha)n^\alpha/2\right)}{M
		\sqrt{\log_+(n)}}\right).
	\]
\end{lemma}

\begin{proof}
	Because $H(t\,,x)-H_n(t\,,x) = \int_{(0,t_{n+1})\times\R}G_{t-s}(y-x)\,W(\d s\,\d y)$
	for all  $n\in\N$, $t\in[t_{n+1}\,,t_n)$ and $x\in\R$, the Wiener isometry implies that
	\[
		\E\left(|H(t\,,x)-H_n(t\,,x)|^2\right)=\int_0^{t_{n+1}}\d s\int_{-\infty}^\infty\d y\
		\left[ G_{t-s}(y-x)\right]^2 = \int_{t-t_{n+1}}^t G_{2s}(0)\,\d s,
	\]
	owing to the semigroup property of the heat kernel. Since $G_{2s}(0)=(8\pi s)^{-1/2}$, it follows 
	that
	\[
		\E\left(|H(t\,,x)-H_n(t\,,x)|^2\right)\propto \sqrt{t} - \sqrt{t-t_{n+1}}\asymp
		\frac{t_{n+1}}{\sqrt t}\le \sqrt{t_{n+1}},
	\]
	uniformly for all $t\in[t_{n+1},t_n]$, $n\in\N$, and $x\in\R$. Apply
	\eqref{t/t} to see that
	\begin{equation}\label{Var(H-H)}
		\adjustlimits\sup_{t\in[t_{n+1},t_n]}\sup_{x\in\R}
		\E\left( | H(t\,,x) - H_n(t\,,x)|^2\right)\lesssim \sqrt{t_n}
		\exp\left( -\tfrac12 (1+\alpha) n^{\alpha}  \right),
	\end{equation}
	uniformly for every $n\in\N$. Thanks to this and a metric entropy argument 
	\cite{L1996}, we can find constants $L_1$ and $L$ such that
	\begin{align*}
		&\E  \| H - H_n \|_{C([t_{n+1},t_n]\times[-1,1])} \\\notag
		&\le L_1
			t_n^{1/4} |\log t_n|^{1/2}\exp\left( -\tfrac14 (1+\alpha) n^{\alpha} \right)
			\le L t_n^{1/4} n^{(1+\alpha)/2}\exp\left( -\tfrac14 (1+\alpha) n^{\alpha} \right),
	\end{align*}
	uniformly for every $n\in\N$. Therefore, \eqref{Var(H-H)}
	and concentration of measure \cite{L1996} 
	together ensure that there exists a number $K>0$ such that
	\begin{align*}
		&\P\left\{ \| H - H_n \|_{C([t_{n+1},t_n]\times[-1,1])} \ge 
			Lt_n^{1/4} n^{(1+\alpha)/2}\exp\left( -\tfrac14 
			(1+\alpha) n^{\alpha} \right) +z
			\right\}\\
		&\le2\exp\left( -\frac{z^2}{2\adjustlimits\sup_{t\in[t_{n+1},t_n]}\sup_{x\in\R}
			\E\left( | H(t\,,x) - H_n(t\,,x)|^2\right)}\right)
			\le2\exp\left( -\frac{z^2 \e^{n^{1+\alpha}/2}}{K\sqrt{t_n}}\right),
	\end{align*}
	uniformly for all $z>0$ and $n\in\N$.
	Let $\eta\in(0\,,\delta)$ be an arbitrary number and apply the above with
	$z=\eta^{1/4}\psi(t_n)$
	and appeal to the fact that
	$t_n^{1/4} n^{\alpha/2}\exp\{-n^{1+\alpha}/2\}\ll z$
	for all $n$ large in order to deduce the assertion of the lemma for all sufficiently large $n$.
	We may increase $M$ further, if we need to, in order to see that the lemma's statement is
	valid for every $n\in\N$.
\end{proof}

We are now able to formulate a restricted small-ball estimate, for $H$,
that we shall need shortly.

\begin{lemma}\label{lem:small-ball-H}
	For every $\gamma>0$,
	\[
		\lim_{n\to\infty} \frac{1}{\log n}
		\log\P\left\{ \| H(\cdot\,,0) \|_{C[t_{n+1},t_n]} \le \gamma\psi(t_n)\right\}
		=-\frac{2\lambda(1+\alpha)}{\pi\gamma^4}.
	\]
\end{lemma}

\begin{proof}
	Choose and fix $\gamma>0$. Since
	\[
		\P\left\{ \| H(\cdot\,,0) \|_{C[0,t_n]} \le \gamma\psi(t_n)\right\} \le
		\P\left\{ \| H(\cdot\,,0) \|_{C[t_{n+1},t_n]}\le \gamma\psi(t_n)\right\},
	\]
	and because $\log\log 1/t_n = (1+\alpha)\log n$, Proposition \ref{pr:H} and
	scaling -- see \eqref{scale:H} --  together imply that
	\begin{equation}\label{gamma:1}
		-\frac{2\lambda(1+\alpha)}{\pi\gamma^4}\le \liminf_{n\to\infty}
		\frac{1}{\log n}
		\log\P\left\{ \| H(\cdot\,,0) \|_{C[t_{n+1},t_n]} \le \gamma\psi(t_n)\right\}.
	\end{equation}
	One can obtain a similar bound in the other direction as follows: Owing to Lemma
	\ref{lem:P(H:n+1)}, for every $\delta\in(0\,,\gamma)$ there exists $M>0$ such that uniformly for all $n\in\N$,
	\begin{align*}
		&\P\left\{ \| H(\cdot\,,0) \|_{C[t_{n+1},t_n]} \le \gamma\psi(t_n)\right\} \\
		&\le \P\left\{ \| H(\cdot\,,0) \|_{C[0,t_n]} \le (\gamma+\delta)\psi(t_n)\right\}
			+\P\left\{ \| H(\cdot\,,0) \|_{C[0,t_{n+1}]} \ge \delta\psi(t_n)\right\}\\
		&\le\P\left\{ \| H(\cdot\,,0) \|_{C[0,t_n]} \le (\gamma+\delta)\psi(t_n)\right\}
			+M \exp\left(-  \frac{ \exp\left( (1+\alpha)n^\alpha/2\right)}{M
			\sqrt{\log_+(n)}} \right).
	\end{align*}
	This proves that
	\begin{equation}\label{gamma:2}
		\limsup_{n\to\infty}
		\frac{1}{\log n}
		\log\P\left\{ \| H(\cdot\,,0) \|_{C[t_{n+1},t_n]} \le \gamma\psi(t_n)\right\}
		\le-\frac{2\lambda(1+\alpha)}{\pi(\gamma+\delta)^4}.
	\end{equation}
	The quantity on the left-hand side does not depend on $\delta\in(0\,,\gamma)$. We therefore
	obtain the lemma from \eqref{gamma:1} and \eqref{gamma:2} upon letting $\alpha\downarrow0$.
\end{proof}

When $n\gg1$, the small-ball probability bound of Lemma \ref{lem:small-ball-H}
for the random field $H$ yields an analogous probability bound for the closely related
field $H_n$, viz.,

\begin{lemma}\label{lem:small-ball-Hn}
	For every $\gamma>0$,
	\[
		\lim_{n\to\infty}
		\frac{1}{\log n}
		\log\P\left\{ \| H_n(\cdot\,,0) \|_{C[t_{n+1},t_n]} \le \gamma\psi(t_n)\right\}
		=-\frac{2\lambda(1+\alpha)}{\pi\gamma^4}.
	\]
\end{lemma}

\begin{proof}
	Lemma \ref{lem:Var(H-H)} ensures that for every $0<\delta<\gamma$ there exists
	$M>0$ such that, uniformly for all $n\in\N$,
	\begin{align*}
		&\P\left\{ \| H_n(\cdot\,,0) \|_{C[t_{n+1},t_n]} \le \gamma\psi(t_n)\right\}\\
		&\hskip.2in\le \P\left\{ \| H(\cdot\,,0) \|_{C[t_{n+1},t_n]} 
			\le (\gamma+\delta)\psi(t_n)\right\}\\
		&\hskip1in+ \P\left\{ \| H_n(\cdot\,0) - H(\cdot\,,0) \|_{C[t_{n+1},t_n]}
			\ge \delta\psi(t_n)\right\}\\
		&\hskip.2in\le \P\left\{ \| H(\cdot\,,0) \|_{C[t_{n+1},t_n]} 
			\le (\gamma+\delta)\psi(t_n)\right\}
			+M\exp\left( -\frac{\exp\left( (1+\alpha)n^\alpha/2\right)}{M\sqrt{\log_+(n)}}\right).
	\end{align*}
	Therefore, Lemma \ref{lem:small-ball-H} ensures that
	\begin{align*}
		&\limsup_{n\to\infty}\frac{1}{\log n}
			\log\P\left\{ \| H_n(\cdot\,,0) \|_{C[t_{n+1},t_n]}\le \gamma\psi(t_n)\right\}\\
		&\le \lim_{n\to\infty}\frac{1}{\log n}
			\log \P\left\{ \| H(\cdot\,,0) \|_{C[t_{n+1},t_n]}\le (\gamma+\delta)\psi(t_n)\right\}
			= -\frac{2\lambda(1+\alpha)}{\pi(\gamma+\delta)^4}.
	\end{align*}
	In like manner, we can prove that
	\begin{align*}
		-\frac{2\lambda(1+\alpha)}{\pi(\gamma-\delta)^4} &=\lim_{n\to\infty}\frac{1}{\log n}
			\log\P\left\{ \| H(\cdot\,,0) \|_{C[t_{n+1},t_n]} \le (\gamma - \delta)\psi(t_n)\right\}\\
		&\le \liminf_{n\to\infty}\frac{1}{\log n}
			\log \P\left\{ \| H_n(\cdot\,,0) \|_{C[t_{n+1},t_n]} \le \gamma\psi(t_n)\right\}.
	\end{align*}
	Let $\delta\downarrow0$ in order to deduce the lemma from the preceding two displays.
\end{proof}

With the preceding preliminary results under way, we can now present the following.

\begin{proof}[Proof of Proposition \ref{pr:Chung:H}]
	By the stationarity of $x\mapsto H(\,\cdot\,,x)$, it suffices to prove
	that
	\begin{equation}\label{eq:Chung:H:1}
		\liminf_{\varepsilon\downarrow0}
		\sup_{t\in[0,\varepsilon]} \frac{|H(t\,,0)|}{\psi(\varepsilon)}=
		\left(\frac{2\lambda}{\pi}\right)^{1/4}\qquad\text{a.s.}
	\end{equation}
	The basic properties of the Wiener
	integral ensure that the events
	\[
		\left\{\omega\in\Omega:\ 
		\| H_n(\cdot\,,0) \|_{C[t_{n+1},t_n]}(\omega) \le \gamma\psi(t_n)\right\},
		\qquad n=1,2,\ldots,
	\]
	are independent for every fixed choice of $\gamma>0$. 
	Therefore, Lemma \ref{lem:small-ball-Hn} and a standard appeal to
	the Borel-Cantelli lemma for independent events together yield
	\[
		\liminf_{n\to\infty}\sup_{t\in[t_{n+1},t_n]}\frac{|H_n(t\,,0)| }{\psi(t_n)}= 
		\left(\frac{2\lambda(1+\alpha)}{\pi}\right)^{1/4}\qquad\text{a.s.}
	\]
	Lemma \ref{lem:Var(H-H)} and the Borel-Cantelli lemma together imply that
	\[
		\| H(\cdot\,,0) - H_n(t\,,\cdot) \|_{C([t_{n+1},t_n]\times[-1,1])}
		= o(\psi(t_n))\qquad\text{as $n\to\infty$\quad a.s.}
	\]
	Therefore, we combine the preceding with Lemma \ref{lem:P(H:n+1)} in order to deduce
	the following:
	\[
		\liminf_{\varepsilon\downarrow0}
		\sup_{t\in[0,\varepsilon]} \frac{|H(t\,,0)|}{\psi(\varepsilon)}\le
		\liminf_{n\to\infty}\sup_{t\in[t_{n+1},t_n]}\frac{|H_n(t\,,0)|}{\psi(t_n)} =
		\left(\frac{2\lambda (1+\alpha)}{\pi}\right)^{1/4}\qquad\text{a.s.}
	\]
	Since the left-most quantity is independent of the sequence $\{t_n\}_{n\in\N}$ -- and
	in particular of $\alpha$ -- we let $\alpha\downarrow0$ to see that
	\[
		\liminf_{\varepsilon\downarrow0}
		\sup_{t\in[0,\varepsilon]} \frac{|H(t\,,0)|}{\psi(\varepsilon)}\le
		\left(\frac{2\lambda}{\pi}\right)^{1/4}\qquad\text{a.s.}
	\]
	This proves half of the assertion of \eqref{eq:Chung:H:1}.
	The other half follows readily from Proposition \ref{pr:H},
	the scaling property \eqref{scale:H} of $H$,
	and a direct application of the Borel-Cantelli lemma.
\end{proof}

\section{Proof of Theorem \ref{th:subseq}} \label{sec:proof:thm}

Throughout this section, we choose and fix a real number $\theta>0$, and define
\[
	\mathcal{D}(n) = \cup_{\substack{j\in\Z_+:\\
	0\le j \le \theta 2^n}} \left\{ j 2^{-n} \right\}
	\qquad\text{for all $n\in\Z_+$}, \quad\text{ so that $|\mathcal{D}(n)| \sim \theta 2^n$
	as $n\to\infty$.}
\]
Also here and throughout,
we choose and fix  a second real number $q>0$ and define $\mathcal{D}_q(1),\mathcal{D}_q(2),\ldots$ to be the following
``slowed down'' version of $\mathcal{D}(1),\mathcal{D}(2),\ldots$:
\[
	\mathcal{D}_q(m) = \mathcal{D}(n)\quad\text{whenever $m\in\N$ satisfies
	$2^{n/q} \le m< 2^{(n+1)/q}$}.
\]
Note in particular that:
\begin{compactenum}
	\item $\mathcal{D}_q(m)\subseteq\mathcal{D}_q(m+1)$ for every $m\in\Z_+$; 
	\item $\cup_{m=0}^\infty\mathcal{D}_q(m)$ coincides
		with the set of all dyadic rationals in $[0\,,\theta]$; and 
	\item $|\mathcal{D}_q(m)|\asymp m^q$, uniformly for all $m\in\N$.
\end{compactenum}
We will use these properties, sometimes without explicit mention, in the sequel.
Finally, we choose and fix $\alpha>0$
throughout this section,
and recall the sequence $\{t_n\}_{n\in\N}=\{t_n(\alpha)\}_{n\in\N}$ from \eqref{t_n}.

\begin{proposition}\label{pr:min-P(n)-Hn}
 As $n\to\infty$,
	\begin{align*}
		&\frac{1}{\log n}\log
			\P\left\{ \adjustlimits\min_{x\in\mathcal{D}_q(n)}
			\sup_{t\in[t_{n+1},t_n]}|H(t\,,x)|\le\gamma\psi(t_n)\right\}\\
		&=\frac{1+o(1)}{\log n}\log
			\P\left\{ \adjustlimits\min_{x\in\mathcal{D}_q(n)}
			\sup_{t\in[t_{n+1},t_n]}|H_n(t\,,x)|\le\gamma\psi(t_n)\right\}
			\to -\left(\frac{2\lambda(1+\alpha)}{\pi\gamma^4}-q\right),
	\end{align*}
	provided that $\gamma>0$ renders the above limit negative;
	that is, provided that $\gamma$ satisfies
	\begin{equation}\label{gamma:q}
		0< \gamma < \left( \frac{2\lambda(1+\alpha)}{\pi q}\right)^{1/4}.
	\end{equation}
\end{proposition}

The proof of Proposition \ref{pr:min-P(n)-Hn} requires first taking three preliminary steps
which, in turn, hinge on the introduction of two new objects. Namely,
we define for every $n\in\N$,
\begin{equation}\label{RI}\begin{split}
	\mathcal{R}(n) &= \left[x-\sqrt{t_n|\log t_n|}\,, x+\sqrt{ t_n|\log t_n|}\right],
		\quad\text{and}\\
	I_n(t\,,x) &= \int_{[t_{n+1},t)\times\mathcal{R}(n)} G_{t-s}(y-x)\,W(\d s\,\d y),
\end{split}\end{equation}
for all $(t\,,x)\in[t_{n+1}\,,t_n]\times\R$.
Recall the random fields $\{H_n\}_{n=1}^\infty$ from \eqref{Hn}.
Our next result shows that $H_n$ and $I_n$ are close, on a suitable scale,
and with high probability. The following constitutes the first step of the proof
of Proposition \ref{pr:min-P(n)-Hn}.

\begin{lemma}\label{lem:H-I}
	For every $\delta\in(0\,,1)$ and for every closed interval $J\subset\R$,
	there exists $M=M(\delta\,,\alpha\,,J)>0$ such that
	\[
		\P\left\{ \|H_n - I_n \|_{C([t_{n+1},t_n]\times J)}
		\ge \delta\psi(t_n)\right\} \le 
		M\exp\left( -\frac{\exp\left( n^{1+\alpha}\right)}{M\sqrt{\log_+(n)}}\right),
	\]
	uniformly for all $n\in\N$.
\end{lemma}

\begin{proof}
	Without too much loss in generality we consider only the case that $J=[-1\,,1]$.
	The general case is proved by making simple adjustments to the following.
	
	Thanks to \eqref{Hn} and \eqref{RI},
     \begin{align*}
		&\E\left( |H_n(t\,,x) - I_n(t\,,x)|^2\right)
			= \int_0^{t-t_{n+1}}\d s
			\int_{y\in\R:|y|>2\sqrt{t_n\log(1/t_n)}}\d y
			\left[ G_s(y)\right]^2\\
		&\le \int_0^{t_n-t_{n+1}}\d s
			\int_{y\in\R:|y|>2\sqrt{t_n\log(1/t_n)}}\d y
			\left[ G_s(y)\right]^2\\
		&\propto\int_0^{t_n-t_{n+1}}\frac{\d s}{\sqrt s}
			\int_{y\in\R:|y|>2\sqrt{ t_n\log(1/t_n)}}\d y\ G_s\left( y/\sqrt 2\right)
			\lesssim\int_0^{t_n}\exp\left( -\frac{t_n\log(1/t_n)}{s}\right)\frac{\d s}{\sqrt s},
	\end{align*}%
	uniformly for all $t\in[t_{n+1}\,,t_n]$, $x\in\R$, and $n\in\N$,
	thanks to the well-known fact that $\P\{|X|>r\}\le 2\exp(-r^2/(4s))$ for all $r>0$ if $X$ has
	a centered normal distribution with variance $2s$ for some $s>0$. If $s\le t_n$, then
	$\exp(-t_n\log(1/t_n)/s)\le t_n$. This yields
	\begin{equation}\label{VVV}
		\adjustlimits\sup_{t\in[t_{n+1},t_n]}\sup_{x\in\R}
		\E\left( |H_n(t\,,x) - I_n(t\,,x)|^2\right) 
		\lesssim t_n^{3/2},
	\end{equation}
	valid uniformly for every $n\in\N$. Therefore, a metric entropy argument \cite{L1996}
	yields the following: Uniformly for all $n\in\N$,
	\[
		\E \| H_n - I_n \|_{C([t_{n+1},t_n]\times[-1,1])} 
		\lesssim t_n^{3/4}\sqrt{\log(1/t_n)}
		\lesssim t_n^{1/4} \e^{-n^{1+\alpha}/4},
	\]
	with room to spare.
	Now we apply concentration of measure \cite{L1996} in conjunction with \eqref{VVV}
	in order to see that there exist $K_i=K_i(\alpha)>0$ $[i=1,2]$
	such that, uniformly for all $n\in\N$ and $z>0$,
	\begin{equation}\label{COM}
		\P\left\{ \| H_n - I_n \|_{C([t_{n+1},t_n]\times[-1,1])} 
		\ge K_1t_n^{1/4} \e^{- n^{1+\alpha}/4} + z\right\}
		\le 2\e^{- K_2 z^2/t_n^{3/2}}.
	\end{equation}
	Choose and fix some $\delta\in(0\,,1)$. For all sufficiently large $n\in\N$,
	\[
		K_2t_n^{1/4} \exp\left( -\frac{n^{1+\alpha}}{4}\right)\le
		\frac{\delta}{2}\psi(t_n),
	\]
	and how large depends only on $(\delta\,,\alpha)$. Therefore, we plug 
	$z=\delta\psi(t_n)/2$ into \eqref{COM} in order to conclude the proof.
\end{proof}

Our next lemma provides the second step in the proof of Proposition \ref{pr:min-P(n)-Hn}.

\begin{lemma}\label{lem:small-ball-In}
	For every $\gamma>0$,
	\[
		\lim_{n\to\infty} (\log n)^{-1}
		\log\P\left\{ \| I_n(\cdot\,,0) \|_{C[t_{n+1},t_n]} \le \gamma\psi(t_n)\right\}
		=- 2\lambda(1+\alpha)/(\pi\gamma^4).
	\]
\end{lemma}

\begin{proof}
	The proof of Lemma \ref{lem:small-ball-In} follows the same pattern as did the proof
	of Lemma \ref{lem:small-ball-Hn}, but uses respectively 
	Lemma \ref{lem:small-ball-Hn} and Lemma \ref{lem:H-I} 
	in place of 
 Lemmas \ref{lem:Var(H-H)} and  \ref{lem:small-ball-H}.
	We leave the remaining details to the interested reader.
\end{proof}

The following probability evaluation is the third, and final, preliminary step 
in our proof of Proposition \ref{pr:min-P(n)-Hn}.

\begin{lemma}\label{lem:min-P(n)-In}
	We have
	\[
		\lim_{n\to\infty}\frac{1}{\log n}\log
		\P\left\{ \adjustlimits\min_{x\in\mathcal{D}_q(n)}
		\sup_{t\in[t_{n+1},t_n]}|I_n(t\,,x)|\le\gamma\psi(t_n)\right\}
		= -\left(\frac{2\lambda(1+\alpha)}{\pi\gamma^4}-q\right),
	\]
	for every $\gamma$ that satisfies \eqref{gamma:q}.
\end{lemma}

\begin{proof}
	If $x_1,x_2,\ldots\in\R$ satisfy the following for all distinct $i,j\in\N$,
	\[
		|x_i-x_j| \ge 2\sqrt{t_n|\log t_n|} = 2 \exp\left( -\frac{n^{1+\alpha}}{2}\right)
		n^{(1+\alpha)/2},
	\]
	then $\{I_n(\,\cdot\,,x_i)\}_{i=1}^\infty$ are obtained by integrating white noise
	over disjoint sets. In particular, the above condition on $x_1,x_2,\ldots$ ensures that 
	$\{I_n(\,\cdot\,,x_i)\}_{i=1}^\infty$
	are i.i.d.\ random variables. If $x,y$ are two distinct points in $\mathcal{D}_q(n)$, then
	\[
		|x-y|\ge 2^{-\lfloor q\log_2 n\rfloor}\asymp n^{-q} \gg
		2 \exp\left( -\frac{n^{1+\alpha}}{2}\right)n^{(1+\alpha)/2},
	\]
	valid for all $n$ large, where how large depends only on $(\alpha\,,q)$.
	Thus, we can see that $\{I_n(\,\cdot\,,x)\}_{x\in\mathcal{D}_q(n)}$ is an i.i.d.\ sequence
	and hence, for every $\gamma>0$,
	\begin{align*}
		&\P\left\{ \textstyle\min_{x\in\mathcal{D}_q(n)}
			\sup_{t\in[t_{n+1},t_n]}|I_n(t\,,x)|\le\gamma\psi(t_n)\right\}\\
		&\hskip.2in= 1- \left(1 - \P\left\{ \textstyle
			\sup_{t\in[t_{n+1},t_n]}|I_n(t\,,x)|\le\gamma\psi(t_n)\right\}\right)^{|\mathcal{D}_q(n)|}\\
		&\hskip.2in= 1- \left(1 - \exp\left\{-\frac{2\lambda(1+\alpha)+o(1)}{\pi
			\gamma^4}\log n\right\}\right)^{|\mathcal{D}_q(n)|}
			\qquad[\text{by Lemma \ref{lem:small-ball-In}}],
	\end{align*}
	as $n\to\infty$. Since there exists $C>0$
	since $|\mathcal{D}_q(n)|\ge Cn^q$ for all $n\in\N$, 
	condition \eqref{gamma:q} implies that
	\begin{equation}\label{I:LB}\begin{split}
		&\liminf_{n\to\infty}
			\frac{1}{\log n}\log\P\left\{ \textstyle\min_{x\in\mathcal{D}_q(n)}
			\sup_{t\in[t_{n+1},t_n]}|I_n(t\,,x)|\le\gamma\psi(t_n)\right\}\\
		&\hskip2.5in\ge -\left(\frac{2\lambda(1+\alpha)}{\pi \gamma^4}-q\right).
	\end{split}\end{equation}
	
	Conversely, since $|\mathcal{D}_q(n)|\asymp n^q$ uniformly for $n\in\N$,
	Boole's inequality and the apparent stationarity of $x\mapsto I_n(\,\cdot\,,x)$
	yields the following, valid uniformly for all $n\in\N$:
	\begin{equation}\label{eq:I_n}\begin{split}
		&\P\left\{ \textstyle\min_{x\in\mathcal{D}_q(n)}
			\sup_{t\in[t_{n+1},t_n]}|I_n(t\,,x)|\le\gamma\psi(t_n)\right\}\\
		&\hskip1.7in\lesssim n^q
			\P\left\{ \| I_n(\cdot\,,0) \|_{C[t_{n+1},t_n]} \le\gamma\psi(t_n)\right\}.
	\end{split}\end{equation}
	Therefore, in light of \eqref{I:LB}, it remains to prove that
	\begin{equation}\label{I:UB}
		\limsup_{n\to\infty}(\log n)^{-1}\log\P\left\{ 
		\| I_n(\cdot\,,0) \|_{C[t_{n+1},t_n]}\le\gamma\psi(t_n)\right\}
		\le - \frac{2\lambda(1+\alpha)}{\pi \gamma^4}.
	\end{equation}
	Let us choose and fix some $\delta\in(0\,,\gamma)$, as close to zero as we wish but fixed, and
	appeal to Lemma \ref{lem:H-I} in order to find a constant $M=M(\delta\,,\alpha)>0$ such that
	\begin{align*}
		&\P\left\{ \| I_n(\cdot\,,0) \|_{C[t_{n+1},t_n]}\le\gamma\psi(t_n)\right\}\\
		&\le\P\left\{ \| H_n(\cdot\,,0) \|_{C[t_{n+1},t_n]}\le(\gamma+\delta)\psi(t_n)\right\}
			+ M\exp\left( -\frac{\exp\left( n^{1+\alpha}\right)}{M\sqrt{\log_+(n)}}\right),
	\end{align*}
	uniformly for all $n\in\N$. This and Lemma \ref{lem:small-ball-Hn} together
	imply that
	\[
		\limsup_{n\to\infty}
		\frac{1}{\log n}\log\P\left\{ \| I_n(\cdot\,,0) \|_{C[t_{n+1},t_n]}
		\le\gamma\psi(t_n)\right\}
		\le - \frac{2\lambda(1+\alpha)}{\pi (\gamma+\delta)^4}.
	\]
	Send $\delta\downarrow0$ to deduce \eqref{I:UB} and hence the lemma.
\end{proof}

We are ready to prove  Proposition \ref{pr:min-P(n)-Hn}. 

\begin{proof}[Proof of Proposition \ref{pr:min-P(n)-Hn}]
	Choose and fix $0<\delta<\gamma$, where $\delta$ is fixed but
	small enough to ensure that 
	\begin{equation}\label{gamma:delta:q}
		\gamma+\delta < \left(\frac{2\lambda(1+\alpha)}{\pi q}\right)^{1/4}.
	\end{equation}
	Lemma \ref{lem:H-I} assures us that there exists $M=M(\delta\,,\alpha)>0$ such
	that, uniformly for all $n\in\N$,
	\begin{align}\label{3L}
		&\P\left\{ \textstyle\min_{x\in\mathcal{D}_q(n)}
			\sup_{t\in[t_{n+1},t_n]}|H_n(t\,,x)|\le\gamma\psi(t_n)\right\}\\\notag
		&\le \P\left\{ \textstyle\min_{x\in\mathcal{D}_q(n)}
			\sup_{t\in[t_{n+1},t_n]}|I_n(t\,,x)|\le(\gamma+\delta)\psi(t_n)\right\}
			+M\exp\left( -\frac{\exp\left( n^{1+\alpha}\right)}{M\sqrt{\log_+(n)}}\right)\\\notag
		&\le \P\left\{ \textstyle\min_{x\in\mathcal{D}_q(n)}
			\sup_{t\in[t_{n+1},t_n]}|H_n(t\,,x)|\le(\gamma+2\delta)\psi(t_n)\right\}
			+2M\exp\left( -\frac{\exp\left( n^{1+\alpha}\right)}{M\sqrt{\log_+(n)}}\right).
	\end{align}
	Therefore, Condition \eqref{gamma:delta:q} and Lemma \ref{lem:min-P(n)-In}
	together imply that the quantity in the middle line of \eqref{3L}
	behaves, as $n\to\infty$, as $n^{-\ell(\delta)+o(1)}$, where
	\[
		\ell(\delta)= q-\frac{2\lambda(1+\alpha)}{\pi(\gamma+\delta)^4}.
	\]
	Consequently,
	\begin{align}\label{wow}
		\limsup_{n\to\infty}\frac{1}{\log n}\log\P\left\{ \textstyle\min_{x\in\mathcal{D}_q(n)}
			\sup_{t\in[t_{n+1},t_n]}|H_n(t\,,x)|\le\gamma\psi(t_n)\right\} &\le-\ell(\delta),\\
			\notag
		\liminf_{n\to\infty}\frac{1}{\log n}\log\P\left\{ \textstyle\min_{x\in\mathcal{D}_q(n)}
			\sup_{t\in[t_{n+1},t_n]}|H_n(t\,,x)|\le(\gamma+2\delta)\psi(t_n)\right\} &\ge-\ell(\delta).
	\end{align}
	Send $\delta\downarrow0$ such that \eqref{gamma:delta:q} holds. 
	The first line of \eqref{wow} yields the following, valid under
	Condition \eqref{gamma:q} alone:
	\[
		\limsup_{n\to\infty}\frac{1}{\log n}\log\P\left\{ \textstyle\min_{x\in\mathcal{D}_q(n)}
		\sup_{t\in[t_{n+1},t_n]}|H_n(t\,,x)|\le\gamma\psi(t_n)\right\} \le-\ell(0).
	\]
	And we can set $\gamma'=\gamma+2\delta$
	to deduce from  the second line in \eqref{wow}  that
	\[
		\liminf_{n\to\infty}\frac{1}{\log n}\log\P\left\{ \textstyle\min_{x\in\mathcal{D}_q(n)}
		\sup_{t\in[t_{n+1},t_n]}|H_n(t\,,x)|\le\gamma'\psi(t_n)\right\} \ge-\ell(\delta),
	\]
	for every pair $(\gamma',\delta)$ that satisfies
	\[
		2\delta<\gamma'<3\delta+\left(\frac{2\lambda(1+\alpha)}{\pi q}\right)^{1/4}.
	\]
	Once again send $\delta\downarrow0$ to deduce from the preceding effort the following:
	For every $\gamma>0$ that satisfies \eqref{gamma:q},
	\[
		\lim_{n\to\infty}\frac{1}{\log n}\log\P\left\{ \textstyle\min_{x\in\mathcal{D}_q(n)}
		\sup_{t\in[t_{n+1},t_n]}|H_n(t\,,x)|\le\gamma\psi(t_n)\right\} = -\ell(0).
	\]
	To complete the proof, we rehash the above argument but replace the role of 
	the ordered pair $(H_n\,,I_n)$ with that of $(H\,,H_n)$ and use Lemma \ref{lem:Var(H-H)}
	instead of Lemma \ref{lem:H-I}.
	We leave the remaining details to the interested reader.
\end{proof}
Recall that $\rightsquigarrow$ denotes subsequential convergence. With that in mind,
we have the following which is a stronger form of \eqref{LIL:1} when $\sigma\equiv1$
and the SPDE is on the line rather than the torus.

\begin{lemma}\label{lem:H:dense}
	With probability one, the random set
	\[
		\left\{x\in\R:\, \sup_{t\in[0,t_n]}\frac{|H(t\,,x)|}{\psi(t_n)}
		\rightsquigarrow \left(\frac{2\lambda(1+\alpha)}{\pi(1+q)}\right)^{1/4}
		\text{ as $n\to\infty$}\right\}
	\]
	is dense in $\R$.
\end{lemma}

\begin{proof}
	Choose and fix two numbers
	\begin{equation}\label{mumu}
		0<\rho_1<1<\rho_2,
	\end{equation}
	and define
	\begin{equation}\label{gammagamma}
		\gamma_i = \left(\frac{2\lambda(1+\alpha)}{\pi(\rho_i+q)}\right)^{1/4}
		\quad\text{for $i=1,2$}.
	\end{equation}
	Note that $\gamma_1 > \gamma_2$.
	
	Consider next the events $E_{1,i}, E_{2,i},\ldots$ $[i=1,2]$, where
	for every $n\in\N$,
	\begin{align*}
		E_{n,1} &= \left\{ \omega\in\Omega:\,\textstyle\min_{x\in\mathcal{D}_q(n)}
			\sup_{t\in[t_{n+1},t_n]}|H_n(t\,,x)|(\omega) <\gamma_1\psi(t_n)\right\},
			\quad\text{and}\\	
		E_{n,2} &= \left\{ \omega\in\Omega:\,\textstyle\min_{x\in\mathcal{D}_q(n)}
			\sup_{t\in[t_{n+1},t_n]}|H_n(t\,,x)|(\omega) \le\gamma_2\psi(t_n)\right\}.
	\end{align*}
	Thanks to \eqref{Hn} and basic properties of Wiener integrals,
	the events $\{E_{n,1}\}_{n=1}^\infty$ are independent (say). Moreover,
	Proposition \ref{pr:min-P(n)-Hn} tells us that for $i=1,2$,
	\[
		\P(E_{n,i}) = n^{-\rho_i+o(1)}\qquad\text{as $n\to\infty$}.
	\]
	Therefore, \eqref{mumu} and a standard appeal to the Borel-Cantelli lemma 
	together yield the following:
	\begin{equation}\label{P(E)P(E)}
		\P\left( \cap_{n=1}^\infty\cup_{l=n}^\infty E_{l,1}
		\right)=1
		\qquad\text{and}\qquad
		\P\left( \cap_{n=1}^\infty\cup_{l=n}^\infty  E_{l,2}\right) = 0.
	\end{equation}
	Now consider the random sets defined by
	\begin{equation}\label{A_n}
		A_n\left(\rho_1\,,\rho_2\right)
		=\left\{x\in\R:\ \gamma_2\psi(t_n) <
		\sup_{t\in[t_{n+1},t_n]}|H_n(t\,,x)|(\omega) <\gamma_1\psi(t_n)
		\right\},
	\end{equation}
	for all $n\in\N$ and $\rho_1,\rho_2$ that satisfy \eqref{mumu}.
	Then, \eqref{P(E)P(E)} says that, with probability one,
	\[
		A_n(\rho_1\,,\rho_2)\cap\mathcal{D}_q(n)\neq\varnothing\quad\text{for infinitely many $n\in\N$.}
	\]
	Since $\mathcal{D}_q(n)\subset\mathcal{D}_q(n+1)$, this implies that 
	\[
		\cup_{k=n}^\infty A_k(\rho_1\,,\rho_2)\cap
		\cup_{m=1}^\infty\mathcal{D}_q(m) \neq\varnothing\quad
		\text{for infinitely many $n\in\N$, almost surely.}
	\]
	And because $\cup_{m=1}^\infty\mathcal{D}_q(m)$ coincides
	with the collection of
	all dyadic rationals in $[0\,,\theta]$, it follows that
	\[
		\cup_{k=n}^\infty 
		A_k(\rho_1\,,\rho_2) \cap[0\,,\theta]\neq\varnothing\quad
		\text{for infinitely many $n\in\N$, almost surely.}
	\]
	Because the random field $x\mapsto H(\,\cdot\,,x)$ is stationary,
	the above continues to hold if we replace $[0\,,\theta]$ by any non-random,
	bounded, open interval $J\subset\R$. This implies in turn that, with probability one,
	\[
		\cup_{k=n}^\infty   A_k(\rho_1\,,\rho_2) 
		\cap J\neq\varnothing\quad
		\text{i.o., $\forall$ bounded
		open interval $J\subset\R$ with rational ends, }
	\]
	where ``i.o.'' denotes ``infinitely often,'' and refers to the occurrence of the event in question
	for infinitely-many [random] $n\in\N$. A consequence of this is that, with probability one,
	\[
\cup_{k=n}^\infty A_k(\rho_1\,,\rho_2)
		\text{ is dense in $\R$ i.o.}
	\]
	Since the [random] set $\cup_{k=n}^\infty A_k(\rho_1\,,\rho_2)$ is open
	for every $n\in\N$, the Baire category theorem ensures that, with probability one,
	\begin{equation}\label{capcap}
\cap_{\substack{(\rho_1,\rho_2)\in\mathbb{Q}^2_+:\\
		0<\rho_1<1<\rho_2}}\cap_{n=1}^\infty
		\cup_{k=n}^\infty  A_k(\rho_1\,,\rho_2)
		\text{ is dense in $\R$}.
	\end{equation}
	Thanks to \eqref{gammagamma}, we have proved that with probability one,
	\[
		\left\{x\in\R:\, \sup_{t\in[t_{n+1},t_n]}\frac{|H_n(t\,,x)|}{\psi(t_n)}
		\rightsquigarrow C^{1/4}
		\text{ as $n\to\infty$}\right\}
		\text{ is dense in $\R$,}
	\]
	where $C=\frac{2\lambda(1+\alpha)}{\pi(1+q)}$.
	Therefore, Lemma \ref{lem:Var(H-H)}, and a standard
	appeal to the Borel-Cantelli lemma together yield the following
	a.s.\ statement:
	\begin{equation}\label{pre}
		\left\{x\in\R: \sup_{t\in[t_{n+1},t_n]}\frac{|H(t\,,x)|}{\psi(t_n)}
		\rightsquigarrow C^{1/4}
		\text{ as $n\to\infty$} \right\}\text{ is dense in $\R$.}
	\end{equation}
	Yet another appeal to the Borel-Cantelli lemma, this time in conjunction with
	Lemma \ref{lem:P(H:n+1)}, implies that with probability one
	$\sup_{t\in[0,t_{n+1}]}\sup_{x\in[-1,1]}
	|H(t\,,x)| = o(\psi(t_n))$ as $n\to\infty$.
	This and \eqref{pre} together  yield the lemma.
\end{proof}

The following verifies a stronger form of \eqref{LIL:2} when $\sigma\equiv 1$
and the SPDE is on $\R$ rather than $\T$.

\begin{lemma}\label{lem:H:dense2}
	With probability one, the random set
	\begin{equation}\label{eq:main}
		\left\{x\in\R:\,\liminf_{\varepsilon\downarrow0}
		\sup_{t\in[0,\varepsilon]} \frac{|H(t\,,x)|}{\psi(\varepsilon)} 
		= \left(\frac{2\lambda }{\pi(1+q)}\right)^{1/4}
		\text{ as $\varepsilon\to 0$}  \right\}
	\end{equation}
	is dense in $\R$.
\end{lemma}

\begin{proof} 
	The proof is similar to that of Lemma \ref{lem:H:dense},
	but requires making a number of subtle changes that we describe
	next. Perhaps most notably,
	and in contrast with the proof of Lemma \ref{lem:H:dense},
	we will use different sequences for the upper and the
	lower bounds on the supremum of $|H|$. 
	
	For the upper bound, we follow the proof of Lemma \ref{lem:H:dense}
	and let $\gamma_1$ be as was defined
	in \eqref{gammagamma} where $\rho_1 \in (0\,,1)$, but rather
	than use the random sets $A_n$ from \eqref{A_n},
	we define new random sets $\tilde{A}_n$ as follows: 
	\begin{equation}\label{A_n(rho1)}
		\tilde{A}_n\left(\rho_1\,,\alpha\right)
		:=\left\{x\in\R:\ \textstyle
		\sup_{t\in[0\,,t_n]}|H(t\,,x)|(\omega) <\gamma_1\psi(t_n)
		\right\}. 
	\end{equation}
	We are including the parameter $\alpha$, inherited through the choice
	of the sequence $\{t_n\}_{n\in\N}$ [see \eqref{t_n}], 
	for reasons that will become manifest 
	soon. We follow closely the proof of Lemma \ref{lem:H:dense}
	in order to find that with probability one,
	\[
		\cup_{k=n}^\infty \tilde{A}_k(\rho_1\,,\alpha)
		\quad \text{is dense in $\R$ i.o.}.
	\]
	Next, we introduce  a sequence that is notably
	distinct  from $\{t_n\}_{n=1}^\infty$: First, choose and fix two
	numbers  $c$ and $\rho_2$ that satisfy
	\[
		0< c <1 <\rho_2,
	\] 
	and define 
	\[
		s_n=c^n\quad  \text{and} \quad 
		\gamma_2 = \left(\frac{2\lambda}{\pi(\rho_2+q)}\right)^{1/4}.
	\] 
	Consider the event $\tilde E_1, \tilde E_2,\cdots$ where 
	\[ 
		\tilde E_{n} = \left\{ \omega\in\Omega:\,\textstyle\min_{x\in\mathcal{D}_q(n)}
		\sup_{s\in[0\,,s_n]}|H(s\,,x)|(\omega) \le\gamma_2 \psi(s_n)\right\}
		\quad \text{for every $n\in\N$.}
	\]
	As was done in \eqref{eq:I_n}, we may appeal to Lemma \ref{lem:small-ball:H}
	in order to deduce that 
	\[
		\P(\tilde E_n) = n^{-\rho_2+o(1)}\qquad\text{as $n\to\infty$}.
	\]
	Therefore, the Borel-Cantelli lemma yields
	\[
		\P\left( \cap_{n=1}^\infty\cup_{l=n}^\infty  \tilde E_l \right) = 0.
	\]
	Define random sets $B_1(\rho_2), B_2(\rho_2),\ldots\subset\R$ via
	\[
		B_n(\rho_2): =\left\{x\in\R:\ \gamma_2\psi(s_n) <\textstyle
		\sup_{t\in[0\,,s_n]}|H(t\,,x)|\right\}\quad\text{for $n\in\N$}.
	\] 
	Then a similar argument as the proof of Lemma \ref{lem:H:dense} 
	shows us that, with probability 1, 
	\[
	\cup_{l=m}^\infty B_l(\rho_2)
			\quad
			\text{is dense in $\R$ i.o.}
		\]
	If there is a realization [$\omega\in\Omega$] for which the random open sets 
	$\cup_{k=n}^\infty \tilde{A}_k(\rho_1\,,\alpha)$ and $\cup_{l=m}^\infty B_l(\rho_2)$ 
	are dense, then, for that very realization, the random set
	$\{ \cup_{k=n}^\infty \tilde{A}_k(\rho_1)\}  \cap 
	\{ \cup_{l=m}^\infty B_l(\rho_2)\}$ is dense in $\R$ 
	thanks to the Baire category theorem and the fact that 
	$\cup_{k=n}^\infty \tilde{A}_k(\rho_1\,,\alpha)$ and $\cup_{l=m}^\infty B_l(\rho_2)$
	are open sets for every $n, m \geq 1$. 
	Therefore, we may apply the Baire category theorem, in much the same way as we just did, 
	one more time in order to establish  that, with probability 1, 
	\[
		\cap_{\alpha\in\mathbb{Q}_+}
		\cap_{\substack{(\rho_1,\rho_2)\in\mathbb{Q}^2_+:\\
		0<\rho_1<1<\rho_2}}\cap_{n=1, m=1}^\infty
		\left\{ \cup_{k=n}^\infty \tilde{A}_k(\rho_1\,,\alpha)\right\}
		\cap \left\{ \cup_{l=m}^\infty B_l(\rho_2)\right\}
		\text{ is dense in $\R$}.
	\]
	We can deduce \eqref{eq:main} from the above, once we unpack the preceding. 
\end{proof}

\begin{proof}[Proof of Theorem \ref{th:subseq}]
	First, let us observe that
	with probability one, the random set
	\begin{equation}\label{Z:dense}
		\left\{x\in\T:\!\!
		\sup_{t\in[0,t_n]}\frac{|Z(t\,,x)|}{\psi(t_n)}
		\rightsquigarrow \left(\frac{2\lambda(1+\alpha)}{\pi(1+q)}\right)^{1/4}
		\!\!\text{as $n\to\infty$} \right\} \text{ is dense in $\T$.}
	\end{equation}
	Indeed, we may appeal to Lemma \ref{lem:H-Z:sup}
	\[
		\varlimsup_{\varepsilon\downarrow0}\sqrt{\frac{\varepsilon}{\log|\log\varepsilon|}}
		\log\P\left\{ \|H-Z\|_{C([0,\varepsilon]\times\T)} \ge 
		\left( \frac{\delta\varepsilon}{\log|\log\varepsilon|}\right)^{1/4}\right\}
		\le -\frac{\sqrt\delta}{10}.
	\]
	In turn, this inequality and a standard application
	of the Borel-Cantelli lemma together imply that
	$\|H-Z\|_{C([0,\varepsilon]\times\T)} =o(\psi(\varepsilon))$ a.s.\
	as $\varepsilon\downarrow0$. Therefore, Lemma \ref{lem:H:dense} implies \eqref{Z:dense}.
	
	Set
	\[
		\chi=\frac{2\lambda(1+\alpha)}{\pi(1+q)},
	\]
	and observe that $\chi$ can take any value in $(0\,,\infty)$.
	This is because the numbers $q>0$ and $\alpha>0$
	[see \eqref{t_n}] can be chosen otherwise arbitrarily.
	
	Next we use Corollary \ref{cor:localize} in order
	to deduce from \eqref{Z:dense} that
	\[
		\left\{x\in\T: \textstyle
		\sup_{t\in[0,t_n]} | u(t\,,x)-(p_t*u_0)(x)| / \psi(t_n)
		\rightsquigarrow\chi^{1/4} |\sigma(u_0(x))| 
		\text{ as $n\to\infty$}\right\}
	\]
	is dense in $\T$. Because $u_0$ is Lipschitz continuous, this implies that almost surely,
	\[
		\left\{x\in\T: 
		\textstyle\sup_{t\in[0,t_n]} | u(t\,,x)-u_0(x)| / \psi(t_n)
		\rightsquigarrow \chi^{1/4}|\sigma(u_0(x))|
		\text{ as $n\to\infty$} \right\}
	\]
	is dense in $\T$. This proves \eqref{LIL:1} of Theorem \ref{th:subseq} in the case
	that $\chi\in(0\,,\infty)$. And when $\chi\in(0\,,2\lambda/\pi)$,
	the very same argument works to prove \eqref{LIL:2}, except we appeal
	to Lemma \ref{lem:H:dense2} in place of Lemma \ref{lem:H:dense} everywhere
	and make adjustments for the change accordingly.
	
For the proof of \eqref{LIL:2},
	the case  $\chi=2\lambda/\pi$ is covered already by Corollary \ref{cor:Chung}.
	For the proof of \eqref{LIL:1}, the cases where $\chi=0$ and $\chi=\infty$ remain
	to be verified; all else has been proved so far. 
	The remaining two cases are handled analogously by making adjustments to the preceding
	arguments. Therefore, we will describe the changes for the proof of \eqref{LIL:1}
	in the case that $\chi=0$ and leave the requisite argument
	for the remaining case [\eqref{LIL:1} when $\chi=\infty$] to the interested reader. 
	
	Choose and fix some $\rho_1\in(0\,,1)$
	and define, in analogy with \eqref{A_n},
	\[
		\bar{A}_n(q) = \left\{ x\in\R:\, \textstyle
		\sup_{t\in[t_{n+1},t_n]} |H_n(t\,,x)|(\omega)
		<\gamma_1\psi(t_n)\right\},
	\]
	where now we are emphasizing the dependence of $\bar{A}_n$ on $q$ and not $\rho_1$.
	Thanks to \eqref{P(E)P(E)}, another category argument yields the following adaptation of
	\eqref{capcap}:
	\[
		\cap_{\substack{q>0\\
		q\in \mathbb{Q}}}\cap_{n=1}^\infty
		\cup_{k=n}^\infty  \bar{A}_k(q)
		\neq\varnothing\quad
		\text{is dense in $\R$}.
	\]
	Therefore, we obtain, using the same argument as before, the following adaptation of \eqref{pre}:
	\[
		\left\{x\in\R:\,\textstyle\sup_{t\in[t_{n+1},t_n]}
		 |H(t\,,x)| / \psi(t_n)
		\rightsquigarrow 0
		\text{ as $n\to\infty$} \right\}
		\text{ is dense in $\R$ a.s.,}
	\]
	and hence
	\[
		\left\{x\in\R:\, \textstyle\sup_{t\in[0,t_n]} |H(t\,,x)| / \psi(t_n)
		\rightsquigarrow 0
		\text{ as $n\to\infty$} \right\}
		\text{ is dense in $\R$ a.s.,}
	\]
	thanks to the same argument that was used at 
	the very end of the proof of Lemma \ref{lem:H:dense}. We now go through the proof
	of Theorem \ref{th:subseq} line by line, making only very small changes to adapt
	the argument, in order to formally justify 
	setting $q=\infty$ in order to finish the proof of the case where $\chi=0$. This completes
	our presentation.
\end{proof}

\begin{spacing}{0.1}
\footnotesize
\noindent {\bf Davar Khoshnevisan.} Department of Mathematics, University of Utah,
	Salt Lake City, UT 84112-0090, USA,
	\texttt{davar@math.utah.edu}\\[.2cm]
\noindent {\bf Kunwoo Kim.} Department of  Mathematics, Pohang University of 
	Science and Technology (POSTECH), 	Pohang, Gyeongbuk, Korea 37673, 
	\texttt{kunwoo@postech.ac.kr}\\[.2cm]
\noindent\textbf{Carl Mueller.} Department of Mathematics, University of Rochester,
	Rochester, NY 14627, USA,\\
	\texttt{carl.e.mueller@rochester.edu}
\end{spacing}
\bigskip

\end{document}